\documentclass[reqno,11pt]{amsart}
\usepackage{amssymb}
\usepackage{verbatim}
\newif\ifpdf
\ifpdf
  \usepackage[pdftex]{graphicx}
  \usepackage[pdftex]{hyperref}
\else
  \usepackage{graphicx}
\fi
\numberwithin{equation}{section}       
 \theoremstyle{plain}    
 \newtheorem{thm}{Theorem}[section]
 \numberwithin{equation}{section} 
 \numberwithin{figure}{section} 
 \theoremstyle{plain}
 \theoremstyle{plain}    
 \newtheorem{cor}[thm]{Corollary} 
 \theoremstyle{plain}    
 \newtheorem{prop}[thm]{Proposition} 
 \theoremstyle{plain}    
 \newtheorem{lem}[thm]{Lemma} 
 \theoremstyle{remark}
 \newtheorem{rem}[thm]{Remark}
 \theoremstyle{definition}
 
\newtheorem*{thmA}{Theorem A} 
\newtheorem*{thmB}{Theorem B} 
\newtheorem*{thmC}{Theorem C} 
\newtheorem*{thmD}{Theorem D}
 \newtheorem*{corA}{Corollary A}

\theoremstyle{plain}
\newtheorem{defi}[thm]{Definition}

\newtheorem*{ackn}{Acknowledgements}

 \newcommand{\A}{{\mathbb{A}}}
\newcommand{\C}{{\mathbb{C}}}
\newcommand{\N}{{\mathbb{N}}}
\newcommand{\PP}{{\mathbb{P}}}
\newcommand{\Q}{{\mathbb{Q}}}
\newcommand{\R}{{\mathbb{R}}}
\newcommand{\T}{{\mathbb{T}}}
\newcommand{\Z}{{\mathbb{Z}}}

\newcommand{\cB}{{\mathcal{B}}}

\newcommand{\cD}{{\mathcal{D}}}
\newcommand{\cE}{{\mathcal{E}}}
\newcommand{\cF}{{\mathcal{F}}}

\newcommand{\cL}{{\mathcal{L}}}

\newcommand{\cO}{{\mathcal{O}}}

\newcommand{\cK}{{\mathcal{K}}}

\renewcommand{\a}{\alpha}
\renewcommand{\b}{\beta}

\newcommand{\de}{\delta}
\newcommand{\e}{\varepsilon}

\newcommand{\MA}{\mathrm{MA}\,}

\newcommand{\id}{\operatorname{id}}

\newcommand{\vol}{\operatorname{vol}}
\newcommand{\tr}{\operatorname{tr}}

\newcommand{\res}{\operatorname{Res}}

\newcommand{\supp}{\operatorname{supp}}

\newcommand{\eq}{{\mu_\mathrm{eq}}}
\newcommand{\eneq}{{\cE_\mathrm{eq}}}
\newcommand{\ena}{{\cE^\A_\mathrm{eq}}}
\newcommand{\ela}{{\cL^\A_k}}

\begin{document}

\setcounter{tocdepth}{1}

\title{Growth of balls of holomorphic sections and energy at equilibrium}

\date{25 march 2010}

\author{Robert Berman, S{\'e}bastien Boucksom}

\address{Chalmers Techniska H{\"o}gskola \\
 G{\"o}teborg\\
 Sweden}

\email{robertb@chalmers.se}

\address{CNRS-Universit{\'e} Paris 7\\
 Institut de Math{\'e}matiques\\
 F-75251 Paris Cedex 05\\
 France}

\email{boucksom@math.jussieu.fr}

\begin{abstract}
Let $L$ be a big line bundle on a compact complex manifold $X$. Given a non-pluripolar compact subset $K$ of $X$ and a continuous Hermitian metric $e^{-\phi}$ on $L$, we define the energy at equilibrium of $(K,\phi)$ as the Monge-Amp\`ere energy of the extremal psh weight associated to $(K,\phi)$. We prove the differentiability of the energy at equilibrium with respect to $\phi$, and we show that this energy describes the asymptotic behaviour as $k\to\infty$ of the volume of the sup-norm unit ball induced by $(K,k\phi)$ on the space of global holomorphic sections $H^0(X,kL)$. As a consequence of these results, we recover and extend Rumely's Robin-type formula for the transfinite diameter. We also obtain an asymptotic description of the analytic torsion, and extend Yuan's equidistribution theorem for algebraic points of small height to the case of a big line bundle. 
\end{abstract} 

\maketitle

\tableofcontents

\newpage

\section*{Introduction}
\subsection{The setting}
Let $L$ be a holomorphic line bundle over a compact complex manifold $X$ of dimension $n$. By a \emph{weighted subset} $(K,\phi)$ (resp.~\emph{a weighted measure} $(\mu,\phi)$), we will mean the data of a non-pluripolar compact subset $K$ of $X$ (resp.~a probability measure with non-pluripolar support) together with the weight $\phi$ of a continuous Hermitian metric $e^{-\phi}$ on $L$ (cf. Section~\ref{sec:weights} for more details on the terminology). Using additive notation for tensor powers, we can then endow the space of global sections $s\in H^0(X,kL)$ of $kL$ with the $L^\infty$-norm
$$\Vert s\Vert_{L^\infty(K,k\phi)}:=\sup_K|s|_{k\phi}$$
and the $L^2$-norm 
$$\Vert s\Vert^2_{L^2(\mu,k\phi)}:=\int_X|s|^2_{k\phi}d\mu,$$ 
both of which are indeed norms under the standing assumption that $\supp\mu$ (resp. $K$) are non-pluripolar. 
Consider the special case where $K$ and $\supp\mu$ are compact subsets of 
$$\C^n\subset\PP^n=:X$$ endowed with the ample line bundle $\cO(1)=:L$.  Restricting to $\C^n$ identifies $H^0(\PP^n,\cO(k))$ with the space of polynomials of total degree at most $k$. The linear form $X_0\in H^0(\PP^n,\cO(1))$ cutting out the hyperplane at infinity induces a singular Hermitian metric on $\cO(1)$ with weight $\log|X_0|$, whose restriction to $\C^n$ is smooth. A continuous weight $\phi$ on $\cO(1)$ defined near $K$ is thus naturally identified with a continuous \emph{function} $(\phi-\log|X_0|)|$ with compact support in $\C^n$. On the other hand a plurisubharmonic (psh for short) function on $\C^n$ with at most logarithmic growth at infinity gets identified with the weight $\phi$ of a non-negatively curved (singular) Hermitian metric on $L$, which will thus be referred to as a \emph{psh weight} (note that the corresponding log-homogeneous function on $L^*$ is psh). 

In the general setting described above, the asymptotic study as $k\to\infty$ of $H^0(X,kL)$ endowed with the above $L^2$ or $L^\infty$-norms thus appears as a natural generalisation of the classical theory of orthogonal polynomials (cf.~for instance~\cite{ST} and in particular Bloom's appendix therein). 

These two norms on $H^0(kL)$ are equivalently described by their unit balls, which will respectively be denoted by
$$\cB^2(\mu,k\phi),\,\cB^\infty(K,k\phi)\subset H^0(kL).$$
The main goal of the present paper is to study the asymptotic behaviour of the volume of these balls as $k\to\infty$. 
As we shall see, it is related to a well-known energy functional that we now describe.

\subsection{The Monge-Amp\`ere energy functional}
We denote the curvature $(1,1)$-form of a smooth weight $\phi$ on $L$ as $dd^c\phi$, and define the \emph{Monge-Amp\`ere operator} on such weights as 
$$\MA(\phi):=(dd^c\phi)^n.$$
We have normalised as usual the operator $d^c$ so that $dd^c=\frac{i}{\pi}\partial\overline{\partial}$. 

Integrating against this measure-valued operator induces a $1$-form on the (affine) space of smooth weights on $L$, and it is a remarkable fact that this $1$-form is \emph{closed}, hence exact. The primitive of this Monge-Amp\`ere $1$-form will be denoted by $\phi\mapsto\cE(\phi)$ and called the \emph{Monge-Amp\`ere} energy functional. It is therefore characterised by the property
\begin{equation}\label{equ:direc}\frac{d}{dt}_{t=0}\cE((1-t)\phi_1+t\phi_2)=\int_X(\phi_2-\phi_1)\MA(\phi_1).\end{equation}
As is the case for any primitive, $\cE$ is only defined up to a constant. We will always assume that it is normalised by $\cE(\phi_0)=0$ for some auxiliary weight $\phi_0$ fixed once and for all. On the other hand, \emph{differences} $\cE(\phi)-\cE(\psi)$ are intrisically defined. An explicit formula for $\cE$ can be obtained by integration along line segments, which yields

\begin{equation}\label{equ:energy}\cE(\phi)-\cE(\psi)=\frac{1}{n+1}\sum_{j=0}^n\int_X(\phi-\psi)(dd^c\phi)^j\wedge(dd^c\psi)^{n-j}.\end{equation}
Note that the right-hand side involves the Bott-Chern secondary class  attached to the Chern character. The functional $\cE$ seems to have been first explicitly mentioned in an article in~\cite{Mab}, where it is denoted by $L$. It is closely related to the $J$-functional of~\cite{Aub}, and it also corresponds to the functional $I$ in~\cite{Che,Don1} and to minus $F_{\omega_{0}}^{0}$ on p.59 of Tian's book~\cite{Tia2}, where it is proved that $\phi\mapsto \cE(\phi)$ is non-decreasing and concave on smooth psh weights. 

By the fundamental work of Bedford-Taylor, mixed Monge-Amp\`ere operators can be extended to locally bounded psh weights $\phi$. Since the difference of two such weights is a bounded function on $X$, we can use formula (\ref{equ:energy}) to \emph{define} the Monge-Amp\`ere energy $\cE(\phi)$ for a locally bounded weight $\phi$. The proofs of all the above properties, which only rely on integration by parts, are then easily extended to this setting. 

The locally bounded case is good enough for our purposes when $L$ is \emph{ample}. The more general situation of a \emph{big} line bundle is treated in Section~\ref{sec:energy} relying on non-pluripolar products of currents and the appropriate integration-by-parts formula proved in~\cite{BEGZ}. The end result is that $\cE(\phi)$ defined by (\ref{equ:energy}) for two psh weights $\phi,\psi$ with minimal singularities in the sense of Demailly (cf. Section~\ref{sec:big}) still satisfies (\ref{equ:direc}) above. It is non-decreasing and concave, and is continuous along monotonic sequences of such weights. 

\subsection{Asymptotics of ball volumes and energy at equilibrium}
Assume now that $L$ is a big line bundle (which implies that $X$ is Moishezon, i.e.~bimeromorphic to  a projective manifold). Given a weighted subset $(K,\phi)$, its \emph{equilibrium weight} is defined as the following extremal weight:
\begin{equation}\label{equ:equi_weight} 
P_K\phi:=\mathrm{sup}^{*}\{\psi\,\mathrm{psh}\,\mathrm{weight},\,\psi\le\phi\,\,\mathrm{on}\,K\},\end{equation}
where the star denotes upper semi-continuous regularisation.
The equilibrium weight is itself a psh weight with minimal singularities (recall that $K$ is assumed to be non-pluripolar throughout). The \emph{equilibrium measure} of $(K,\phi)$ is the probability measure defined by
\begin{equation}\label{equ:equi_meas}\eq(K,\phi):=\vol(L)^{-1}\MA(P_K\phi).\end{equation}
The normalising factor is the \emph{volume} of $L$, i.e. 
$$\vol(L)=\lim_{k\to\infty}\frac{n!}{k^n}N_k$$
where $N_k:=h^0(kL)$ denotes the complex dimension of $H^0(kL)$ (cf.~Theorem~\ref{thm:mass}). Note that $\vol(L)>0$, precisely because $L$ is big. 

The measure $\eq(K,\phi)$ is concentrated on $K$, and $P_K\phi=\phi$ holds a.e.~on $K$ with respect to this measure (cf.~Proposition~\ref{prop:support}). We define the \emph{energy at equilibrium} of $(K,\phi)$ as
\begin{equation}\label{equ:eneq}\eneq(K,\phi):=\vol(L)^{-1}\cE(P_K\phi).
\end{equation}
The energy at equlibrium is well-defined only up to an overall additive constant, but \emph{differences} 
$$\eneq(K_1,\phi_1)-\eneq(K_2,\phi_2)$$
are intrinsically defined. Our choice of normalisation yields the scaling property
\begin{equation}\label{equ:scaling}\eneq(K,\phi+c)=\eneq(K,\phi)+c
\end{equation}
for each constant $c\in\R$. 

On the other hand we introduce the $\cL$-functionals
\begin{equation}\label{equ:L-func}\cL_k(K,\phi):=\frac{1}{2kN_k}\log\vol_k\cB^\infty(K,k\phi),
\end{equation}
and
\begin{equation}\label{equ:L-funcl2}
\cL_k(\mu,\phi):=\frac{1}{2kN_k}\log\vol_k\cB^2(\mu,k\phi),
\end{equation}
where $\mu$ is a probability measure on $X$ with non-pluripolar support. These functionals are meant to be reminiscent of Donaldson's $\cL$-functionals~\cite{Don1}. The volume $\vol_k$ denotes Lebesgue measure on the vector space $H^0(kL)$, and is thus only defined up to a multiplicative constant. As a consequence, the functionals $\cL_k$ are defined up to overall additive constants, but here again \emph{differences} $\cL_k(K_1,\phi_1)-\cL_k(K_2,\phi_2)$ (resp. $\cL_k(\mu_1,\phi_1)-\cL_k(\mu_2,\phi_2)$) are well-defined since they do not depend on the choice of $\vol_k$. Since $H^0(kL)$ has real dimension $2N_k$, our choice of normalisation yields
 \begin{equation}\label{equ:L-scaling}\cL_k(K,\phi+c)=\cL_k(K,\phi)+c
 \end{equation}
for each constant $c\in\R$ (and similarly with $\mu$ in place of $K$) which should of course be compared to (\ref{equ:scaling}). Equivalently $\cL_k$ defines a single valued function of $(K,\phi)$ relatively to a
fixed reference weighted set, if $\vol_k$ is taken as the Lesbegue measure
which gives a unit mass to the corresponding reference ball.

We now describe our first main result:

\begin{thmA} Let $X$ be a compact complex manifold and $L$ be a big line bundle, let $(K_j,\phi_j)$, $j=1,2$ be two weighted subsets.Then as $k\to\infty$ we have
\begin{itemize}
\item[(i)] 
$$
\cL_k(K_1,\phi_1)-\cL_k(K_2,\phi_2)\to\eneq(K_1,\phi_1)-\eneq(K_2,\phi_2).
$$ 
\item[(ii)] If furthermore $\mu_j$ is a probability measure on $K_j$ with the Bernstein-Markov property with respect to $(K_j,\phi_j)$, $j=1,2$, then we have
$$\cL_k(\mu_1,\phi_1)-\cL_k(\mu_2,\phi_2)\to\eneq(K_1,\phi_1)-\eneq(K_2,\phi_2).$$
\end{itemize}
\end{thmA}
Extending classical terminology we say that a probability measure $\mu$ on $K$ has the \emph{Bernstein-Markov property}Ê with respect to $(K,\phi)$ if the distortion between the $L^\infty(K,k\phi)$-norm and the $L^2(\mu,k\phi)$-norm on $H^0(kL)$ has subexponential growth as $k\to\infty$ (cf. Section~\ref{sec:BM}). Assertion (ii) of Theorem A is a rather direct consequence of (i), but conversely the proof of Theorem A settles as first step the special case of (ii) where the $\phi_j$'s are smooth and the $\mu_j$'s are smooth volume forms. It is indeed an easy consequence of the mean-value inequality that $\mu_j$ has the Bernstein-Markov property with respect to $(X,\phi_j)$ in that case (cf.~Lemma~\ref{lem:BMcont}) - and a much more precise estimate of the distortion is available in that case via Bergman kernels asymptotics. A crucial ingredient in this first step is our second main result:

\begin{thmB} Let $L$ be a big line bundle on a compact complex manifold $X$, and let $K$ be a non-pluripolar compact subset of $X$. Then $\phi\mapsto\eneq(K,\phi)$ is concave and continuous on the space of continuous weights. It is G\^ateaux differentiable, with derivatives given by integration against the equilibrium measure: 
$$\frac{d}{dt}_{t=0}\eneq(K,\phi+t v)=\langle v,\eq(K,\phi)\rangle$$
for every continuous function $v$.
\end{thmB}
This result is a complex analogue of a result of Alexandrov in the setting of convex geometry~\cite{Ale} (see also~\cite{Sch} p.345). It bears a strong resemblance with the differentiability property of the volume of divisors~\cite{BFJ}, which is in some sense a non-archimedean analogue of the present result (compare~\cite{BFJ2}).  

The differentiability property can be understood as a \emph{linear reponse} property for the energy at equilibrium. Theorem B is a key tool in the proof of the arithmetic equidistribution result to be described below (Theorem D). It also found applications in equidistribution theorems for Fekete points and related results~\cite{BB2,BWN,BBWN}, in the proof of a large deviation principle for determinantal point processes~\cite{Ber3,Ber4} as well as in a variational approach to complex Monge-Amp\`ere equations~\cite{BBGZ}.

\subsection{From volumes of $L^2$-balls to transfinite diameters}
Given a basis $S=(s_1,...s_N)$ of $H^0(L)$ let
$$
\det S\in H^0(X^N,L^{\boxtimes N})
$$ 
be the determinant section, locally defined by

$$(\det S)(x_1,...,x_N):=\det(s_i(x_j))_{i,j}.$$
Given a weighted subset $(K,\phi)$ and a probability measure $\mu$ on $K$ the $L^\infty$-norm (resp. $L^2$ norm ) of $\det S$ with respect to the induced probability measure $\mu^N$ on $K^N$ and the induced weight
$$(x_1,...,x_N)\mapsto\psi(x_1)+...+\psi(x_N)$$
on $L^{\boxtimes N}$ will simply be denoted by
$$
\Vert\det S\Vert_{L^\infty(K,\phi)}:=\sup_{(x_1,...,x_N)\in K^N}|\det(s_i(x_j))|e^{-\left(\phi(x_1)+...+\phi(x_N)\right)}
$$
and
$$
\Vert\det S\Vert^2_{L^2(\mu,\phi)}:=\int_{(x_1,...,x_N)\in X^N}|\det(s_i(x_j))|^2e^{-2\left(\phi(x_1)+...+\phi(x_N)\right)}\mu(dx_1)...\mu(dx_N).
$$
In the classical case $(X,L)=(\PP^n,\cO(1))$ we may choose $S_k$ as the set of monomials of degree at most $k$. Given a weighted compact subset $(K,\phi)$ the limit
$$
\lim_{k\to\infty}\Vert\det S_k\Vert_{L^\infty(K,k\phi)}^{1/k^{n+1}}
$$ 
provided it is shown to exist, coincides with Leja's definition of the \emph{transfinite diameter} of $(K,\phi)$ - up to an exponent only depending on $n$. The existence of the limit in the unweighted case was in fact only proved in 1975 by Zaharjuta~\cite{Zah}. 

The basis $S_k$ of monomials is orthonormal with respect to $L^2(\nu,\psi)$, $\nu$ denotes the Haar measure on the compact torus $\T^n\subset\C^n$ and $\psi=\log|X_0|$ denotes the weight on $\cO(1)$ induced by the section cutting out the hyperplane at infinity. Since $\nu$ is known to have the Bernstein-Markov property with respect to $(\T^n,\psi)$ (\cite{NZ}, cf. also Section~\ref{sec:BM}), the next result generalizes in particular Zaharjuta's:

\begin{corA} Let $(E,\psi)$ be a weighted subset and let $\nu$ be a probability measure on $E$ with the Bernstein-Markov property. For each $k$, let $S_k$ be an $L^2(\nu,k\psi)$-orthonormal basis of $H^0(kL)$. 
\begin{itemize}
\item[(i)] For every weighted subset $(K,\phi)$ we have
$$
\lim_{k\to\infty}\frac{1}{k N_k}\log\Vert\det S_k\Vert_{L^\infty(K,k\phi)}=\eneq(E,\psi)-\eneq(K,\psi).
$$ 
\item[(ii)]  If $\mu$ is a probability measure with the Bernstein-Markov property for $(K,\phi)$ then 
$$
\lim_{k\to\infty}\frac{1}{k N_k}\log\Vert\det S_k\Vert_{L^2(\mu,k\phi)}=\eneq(E,\psi)-\eneq(K,\psi).
$$
\end{itemize}
\end{corA}
In the $\C^n$ case, the existence of the limit in (i) in the weighted case was also independently obtained in~\cite{BL3} using~\cite{Rum}. 

Let us quickly explain how Corollary A relates to Theorem A. Since $L^2$-norms are induced by scalar products, ratios of $L^2$-balls can be expressed as Gram determinants:
\begin{equation}\label{equ:gram}\frac{\vol\cB^2(\nu,\psi)}{\vol\cB^2(\mu,\phi)}=\det\left(\langle s_i,s_j\rangle_{L^2(\mu,\phi)}\right)_{i,j},\end{equation}
where $S=(s_1,...,s_{N})$ is an $L^2(\nu,\psi)$-orthonormal basis of $H^0(L)$. On the other hand a row and column expansion of the determinant shows that

\begin{equation}\label{equ:detL2}\Vert\det S\Vert^2_{L^2(\mu,\phi)}=N!\det\left(\langle s_i,s_j\rangle_{L^2(\mu,\phi)}\right)_{i,j}.
\end{equation}

We thus get 
$$
\frac{1}{k N_k}\log\Vert\det S_k\Vert_{L^2(\mu,k\phi)}=\cL_k(\nu,k\psi)-\cL_k(\mu,k\phi)+\frac{1}{2kN_k}\log N_k!
$$
which shows that (ii) of Corollary A is equivalent to (ii) of Theorem A since $\log N_k!=O(k^n\log k)=o(kN_k)$. 

We give in Proposition~\ref{prop:recursion} a recursion formula relating the Monge-Amp\`ere energy on $X$ to that on a hypersurface $Y$. It shows that Corollary A contains in particular Rumely's Robin-type formula for the transfinite diameter in $\C^n$~\cite{Rum}. We also show how to recover DeMarco-Rumely's results ~\cite{DMR} in Section~\ref{sec:resultant}. 

\subsection{Applications to analytic torsion and Arakelov geometry.}
In the last part of the paper, we give two further applications of Theorems A and B related to Arakelov geometry. As a consequence of Theorem A, we will first describe the asymptotic behaviour of the Ray-Singer analytic torsion $T(k\phi)$ of large multiples of a smooth weight $\phi$ with arbitrary curvature (computed with respect to a fixed K\"ahler metric $\omega$), refining results of Bismut-Vasserot~\cite{BV}. More specifically we prove:

\begin{thmC}\label{thm:torsion} If $L$ is an ample line bundle and $\phi$ is a smooth weight on $L$ with arbitrary curvature, then 
$$\lim_{k\to\infty}\frac{n!}{2k^{n+1}}T(k\phi)=\cE(\phi)-\cE(P_X\phi).$$
\end{thmC}

Our second application is a generalisation of Yuan's equidistribution theorem for points of small height~\cite{Yua} to the case of a \emph{big} line bundle (but at archimedean places only). Assume that $X$ is a smooth projective variety defined over a number field, say $\Q$ for simplicity. Let $L$ be a big line bundle on $X/\Q$. Denoting by $\A$ the ad\`eles of $\Q$, $H^0(kL)_\Q$ embeds as a co-compact subgroup of 
$$H^0(kL)_\A\subset H^0(kL)_\R\times\Pi_p H^0(L)_{\Q_p}$$ 
which enables us to normalise the Haar measure $\vol^\A_k$ on $H^0(kL)_\A$ by 
$$\vol^\A_k H^0(kL)_\A/H^0(kL)_\Q=1.$$    

Suppose given a collection $(\phi_p)$ of continuous weights on $L_{\C_p}$ over $X(\C_p)$ for every prime $p$ such that all but finitely of them are induced by a model of $X$ over $\Z$. The superscript $\A$ will be used to indicate that an object implicitly depends on $(\phi_p)$. 

If $\phi$ is a continuous weight on $L_\C$ over $X(\C)$ we define the \emph{adelic unit ball}
\begin{equation}\label{equ:ad_ball}\cB^\A(\phi):=H^0(L)_\A\cap\left(\cB^\infty_\R(\phi)\times\Pi_p\cB^\infty_{\Q_p}(\phi_p)\right)
\end{equation}
and we can then consider the corresponding adelic $\cL$-functionals
\begin{equation}\label{equ:ela}\ela(\phi):=\frac{1}{kN_k}\log\vol^\A_k\cB^\A(k\phi)
\end{equation}
As opposed to the other $\cL$-functionals introduced so far, the adelic $\cL$-functionals $\ela$ are well-defined without any further normalisation issue. 

We now introduce the \emph{adelic energy at equilibrium} as  
$$\ena(\phi):=\limsup_{k\to\infty}\ela(\phi)\in[-\infty,+\infty].$$
The exponential of the right-hand side is called the \emph{sectional capacity} in~\cite{RLV}, where it is proved that the limsup actually is a limit when $L$ ample. Still assuming that $L$ is merely big, Theorems A and B together will enable us to show (Lemma~\ref{lem:arith}) that $\ena(\cdot)$ is differentiable at any weight $\phi$ where it is finite, with derivative given by integration against the equilibrium measure $\eq(X(\C),\phi)$ 

On the other hand, the above data allows to define the \emph{height} $h^\A_\phi(x)$ of any point $x\in X(\overline{\Q})$ (cf.~\ref{equ:height}). If $x_j\in X(\overline{\Q})$ is a \emph{generic} sequence, that is a sequence converging to the generic point of $X$ in the Zariski topology, then it is an easy consequence of the adelic Minkowski theorem (cf. Section~\ref{sec:heights}) that their heights admit the asymptotic lower bound
$$\liminf_{j\to\infty} h^\A_\phi(x_j)\ge\ena(\phi).$$
Following the original variational principle first used by Szpiro, Ullmo and Zhang~\cite{SUZ}, we will prove
\begin{thmD} Using the above notations, supppose that $x_{j}\in X(\overline{\Q})$ is a generic sequence such that
$$\lim_{j\to\infty}h_\phi(x_j)=\ena(\phi)\in\R.$$
Then the Galois orbits of the $x_{j}$'s equidistribute on $X(\C)$ as $j\to\infty$ towards the equilibrium measure $\eq(X(\C),\phi)$.    
\end{thmD}

\subsection{Structure of the paper}
\begin{itemize}
\item Sections~\ref{sec:prelim} and \ref{sec:BM} contain preliminary results on Monge-Amp\`ere operators and Bergman kernels asympotics. 

\item Section~\ref{sec:energy} extends to our singular setting standard facts on the Monge-Amp\`ere energy functional, and contains the proof of Theorem B. 

\item Section~\ref{sec:thmA} contains the proofs of Theorem A and Corollary A followed by a sketch of an alternative argument in the ample case. 

\item Section~\ref{sec:classical} presents applications to the $\C^n$ setting. 

\item Finaly Section~\ref{sec:equi} presents applications to Arakelov geometry, in particular the proof of Theorem D.
\end{itemize}

\begin{ackn} We would like to thank B.~Berndtsson, F.~Berteloot, A.~Chambert-Loir, J.-P.~Demailly, V.~Guedj, C.~Mourougane, N.~Levenberg and A.~Zeriahi for interesting discussions related to the contents of the present paper. We are especially grateful to N.~Levenberg for pointing out a gap in the proof of Theorem B in a previous version of this work. Finally we thank the anonymous referees for some useful suggestions regarding the organization of the article. 
\end{ackn}

\section{\label{sec:prelim}Mixed Monge-Amp{\`e}re operators and equilibrium weights}
The goal of this section is to collect some results on mixed Monge-Amp{\`e}re
operators that are required to study the Monge-Amp\`ere energy functional in the case of a \emph{big} line bundle $L$. 

The reader primarily interested in the case of an \emph{ample} line bundle
will realise that the results we mention are completely standard in
that setting (cf.~for instance~\cite{GZ}), and proofs in the general case can be found in~\cite{BEGZ}

\subsection{Weights vs.~metrics}\label{sec:weights}

Let $X$ be a complex manifold. We will use the additive notation for
the Picard group of line bundles on $X$, that is given line bundles $L,M$ on $X$ we will write $L+M:=L\otimes M$ and $kL:=L^{\otimes k}$. Similarly we want to use an additive notation for singular Hermitian metrics on line bundles. This is formally achieved through the following definition. 

\begin{defi} A \emph{weight}Ê $\phi$ on a line bundle $L$ over $X$ is a locally integrable function on the complement of the zero-section in the total space of the dual line bundle $L^*$ satisfying the log-homogeneity property
$$\phi(\lambda v)=\log|\lambda|+\phi(v)$$
for all non-zero $v\in L^*$, $\lambda\in\C$. 
\end{defi}
Setting
$$|w|_h:=|\langle w,v\rangle| e^{-\phi(v)}$$
for every non-zero vector $w\in L$ (resp. $v\in L^*$) establishes a bijection $\phi\mapsto h$ between the set of weights $\phi$ on $L$ and the set  of singular hermitian metrics $h$ on $L$, and we will simply denote by $h=e^{-\phi}$ the metric on $L$ induced by $\phi$. 

If we let $p:L^*\to X$ be the fibre projection then for every two weights $\phi_1$, $\phi_2$ on $L$ we have $\phi_1-\phi_2=u\circ p$ for a unique function $u\in L^1_{\text{loc}}(X)$. We will simply identify $\phi_1-\phi_2$ with the corresponding function on $X$, so that the set of all weights on $L$ becomes an affine space modelled on $L_{\text{loc}}^{1}(X)$.

A section $s\in H^0(X,L)$ induces a weight on $L$ denoted by
$\log|s|$ and defined by 
$$\log|s|(v):=\log|\langle s,v\rangle|$$ 
for $v\in L^*$. Note that the pointwise length of $s$ in terms of the Hermitian metric $e^{-\phi}$ is equal to $\exp(\log|s|-\phi)$, i.e. we have
$$|s|_\phi=|s|e^{-\phi}.$$

The curvature current of the singular metric $e^{-\phi}$ pulls-back to $dd^c\phi$ under the projection $p:L^*\to X$ and we will somewhat abusively denote by $dd^{c}\phi$ the curvature current on $X$ itself. One must be careful with this suggestive notation, since the curvature current $dd^{c}\phi$
is definitely not \emph{exact} on $X$ in general. We have set as usual
$dd^{c}=:\frac{i}{\pi}\partial\overline{\partial}$ in order to ensure that
the cohomology class of the closed current $dd^{c}\phi$ coincides with the first Chern class $c_{1}(L)\in H^{2}(X,\R)$. With this normalisation the current $dd^c\log|s|$ is equal to the integration current on the zero-divisor of $s$ as a consequence of the Lelong-Poincar\'e formula. 

We will say that a weight $\phi$ is plurisubharmonic (psh for short) if it is psh as a function on the total space $L^*$. The curvature current $dd^c\phi$ is thus a positive (in the French sense of the word, i.e.~non-negative) $(1,1)$-current. This formalism relates to the notion of quasi-psh functions as follows. If $\theta$
is a given closed $(1,1)$-form, a (usc, locally integrable) function
$u$ on $X$ is said to be $\theta$-psh iff $\theta+dd^{c}u\geq0$.
When the cohomology class of $\theta$ is the first Chern class $c_{1}(L)$,
there exists a smooth weight $\phi_{0}$ on $L$, unique up to a constant,
such that $dd^{c}\phi_{0}=\theta$. It follows that $\phi\mapsto u=\phi-\phi_{0}$
establishes a bijection between the set of psh weights $\phi$ on
$L$ and the set of $\theta$-psh functions $u$ on $X$, and we have
$dd^{c}\phi=\theta+dd^{c}u$.

\subsection{Big bundles and minimal singularities}\label{sec:big}
Recall that a line bundle $L$ on a compact complex manifold $X$ is said to be \emph{pseudo-effective} (\emph{psef} for short) iff it admits a psh weight. The line bundle $L$ is said to be \emph{big} iff its \emph{volume} 
$$\vol(L):=\limsup_{k\to\infty}\frac{n!}{k^{n}}h^{0}(kL)$$
 is positive. Here we write as usual by $h^{0}:=\dim H^{0}$, and the $\limsup$ is actually a limit as a consequence of Fujita's theorem. A theorem independently proved by Bonavero~\cite{Bon} and Ji-Shiffmann~\cite{JS} asserts that $L$ is big iff it admits a \emph{strictly psh} weight, i.e.~a singular weight $\phi$ whose curvature current $dd^c\phi$ dominates a (smooth) positive $(1,1)$-form. 
 
 It follows from Demailly's regularisation theorem~\cite{Dem3} that $\phi$ can then be chosen to have analytic singularities, and in particular to be locally bounded on a Zariski open subset $\Omega$ of $X$. Finally note that $X$ is Moishezon, i.e.~bimeromorphic to a projective manifold, iff it admits a big line bundle. 
 
Given two psh weights $\phi_{1},\phi_{2}$ on $L$, one says that $\phi_{1}$
is \emph{more singular} than $\phi_{2}$ if $\phi_{1}\leq\phi_{2}+O(1)$.
As has been observed by Demailly, any pseudo-effective line bundle $L$
admits psh weights with minimal singularities in this sense.  Indeed given a smooth weight $\phi$ on $L$ the equilibrium weight  
$$P_X\phi=\sup\left\{ \psi,\,\psi\,\textrm{psh weight on$\, L$},\,\psi\leq\phi\right\}$$
is automatically (usc and) psh, and it plainly has minimal singularities.
We will at any rate come back to this construction in what follows.

Note that the difference between any two psh weights with minimal
singularities is a bounded function by definition. When $L$ is ample, the psh weights
with minimal singularities are exactly the locally bounded psh weights,
and in the general case the former appear to share many of the nice properties
the latter exhibit in the setting of pluripotential theory.

When $L$ is only big, there exists as we saw a strictly psh weight that is locally bounded on a Zariski open subset $\Omega$
of $X$. It follows that \emph{every} psh weight with minimal singularities on $L$
is locally bounded on this same $\Omega$.

\subsection{Mixed Monge-Amp{\`e}re operators and comparison principle}\label{sec:mixed}
As explained above, results in this section are standard when dealing
with ample line bundles. Indeed, they all follow from Bedford-Taylor's
local results for locally bounded psh weights. The proofs in the general 
situation where line bundles are merely big can be found in~\cite{BEGZ}.

Let $L$ be a big line bundle. By what we saw above, we can choose a Zariski open subset $\Omega$ on which every psh weight with minimal singularities is locally bounded. 

Now let $\phi_{1},...,\phi_{n}$ be psh weights on $L$ that are locally bounded on $\Omega$. We can then define the Bedford-Taylor wedge product
$$dd^{c}\phi_{1}\wedge...\wedge dd^{c}\phi_{n}$$
as a positive measure on $\Omega$. Recall that this is done by locally setting $dd^cu\wedge T:=dd^c(uT)$ whenever $u$ is a locally bounded psh function and $T$ is a closed positive current (which thus has measure coefficients). It was proved by Bedford-Taylor~\cite{BT82} that the resulting measure $dd^c\phi_1\wedge...\wedge dd^c\phi_n$ puts no mass on pluripolar subsets of $\Omega$. The following result is proved in~\cite{BEGZ}.

\begin{thm}\label{thm:mass}
Let $\phi_{1},...,\phi_{n}$ (resp.~$\psi_1,...,\psi_n$) be psh weights on $L$ that are locally bounded on a Zariski open subset $\Omega$. If $\phi_j$ is less singular than $\psi_j$ for all $j$, then we have 
$$\int_\Omega dd^c\psi_1\wedge...\wedge dd^c\psi_n\le\int_\Omega dd^c\phi_1\wedge...\wedge dd^c\phi_n\le\vol(L).$$
Equality holds on the right-hand side when the $\phi_j$'s have minimal singularities. 
\end{thm}
This says in particular that $dd^c\phi_1\wedge...\wedge dd^c\phi_n$ has finite total mass, and we can thus introduce:

\begin{defi} If $\phi_1,...,\phi_n$ are psh weights on $L$ that are locally bounded on a Zariski open subset, the non-pluripolar product 
$$\langle dd^c\phi_1\wedge...\wedge dd^c\phi_n\rangle$$ is defined as the trivial extension to $X$ of the positive measure $dd^{c}\phi_{1}\wedge...\wedge dd^{c}\phi_{n}$ on $\Omega$. In particular, the Monge-Amp\`ere measure of a psh weight $\phi$ locally bounded on a Zariski open subset $\Omega$ is defined by
$$\MA(\phi):=\langle(dd^c\phi)^n\rangle.$$
\end{defi} 
We stress that such non-pluripolar products $\langle dd^c\phi_1\wedge...\wedge dd^c\phi_n\rangle$ put no mass on pluripolar subsets of $X$, and therefore do not depend on the choice of $\Omega$. By Theorem~\ref{thm:mass}, the total mass
$$\int_X\langle dd^c\phi_1\wedge...\wedge dd^c\phi_n\rangle$$
only depends on the singularity classes of the $\phi_j$'s and is equal to $\vol(L)$ when the $\phi_j$'s have minimal singularities. 

The non-pluripolar Monge-Amp\`ere operator so defined satisfies the following generalised comparison principle, which will be a crucial ingredient in the proof of Theorem B. 

\begin{cor}\label{cor:comparison} Let $\phi_1$ and $\phi_2$ be two psh weights on $L$ that are locally bounded on a Zariski open subset. If $\phi_1\le\phi_2+O(1)$, then we have
$$\int_{\{\phi_2<\phi_1\}}\MA(\phi_1)\le\int_{\{\phi_2<\phi_1\}}\MA(\phi_2).$$
\end{cor}
\begin{proof} It is an important result of Bedford-Taylor~\cite{BT87} that $u\mapsto(dd^c u)^n$ is local in the plurifine topology for locally bounded psh functions $u$. By definition of the non-pluripolar Monge-Amp\`ere operator, it follows that $\phi\mapsto\MA(\phi)$ defined above is also local in the plurifine topology (cf.~\cite{BEGZ}). Now let $\e>0$. The psh weight $\max(\phi_2,\phi_1-\e)$ coincides with $\phi_2$ on the plurifine open subset
$\{\phi_2>\phi_1-\e\}$ and with $\phi_1-\e$ on the plurifine open subset
$\{\phi_2<\phi_1-\e\}.$ 
It follows that
$$\int_X\MA(\max(\phi_2,\phi_1-\e))$$
$$\ge\int_{\{\phi_2>\phi_1-\e\}}\MA(\phi_2)+\int_{\{\phi_2<\phi_1-\e\}}\MA(\phi_1)$$
which is in turn
$$\ge\int_X\MA(\phi_2)-\int_{\{\phi_2<\phi_1\}}\MA(\phi_2)+\int_{\{\phi_2<\phi_1-\e\}}\MA(\phi_1).$$
On the other hand Theorem~\ref{thm:mass} yields
$$\int_X\MA(\phi_2)=\int_X\MA(\max(\phi_2,\phi_1-\e))$$ since $\phi_1\le\phi_2+O(1)$ implies
$$\max(\phi_2,\phi_1-\e)=\phi_2+O(1),$$ and the result now follows by monotone convergence by letting $\e\to 0$.
\end{proof}
We infer the following domination principle (cf.~\cite{BEGZ}):
\begin{cor}\label{cor:domination} Let $\phi_1$ and $\phi_2$ be two psh weights on $L$ and suppose that $\phi_2$ has minimal singularities. If $\phi_1\le\phi_2$ holds a.e. wrt $\MA(\phi_2)$, then $\phi_1\le\phi_2$ everywhere on $X$. 
\end{cor}

The following continuity result is proved in~\cite{BEGZ}. 
\begin{thm}
\label{thm:continuous} Let $\psi_{0}$ be a fixed psh weight with
minimal singularities on $L$. Then the measure-valued operators 
$$(\phi_1,...,\phi_{n})\mapsto\langle dd^c\phi_{1}\wedge...\wedge dd^c\phi_{n}\rangle$$
and 
$$(\phi_{0},...,\phi_{n})\mapsto(\phi_{0}-\psi_{0})\langle dd^c\phi_{1}\wedge...\wedge\phi_{n}\rangle$$
are continuous along convergent sequence $\phi_j^{(k)}\to\phi_j$ of psh weights with minimal singularities in the following three cases:
\begin{itemize} 
\item $\phi_j^{(k)}$ decreases pointwise to $\phi_j$.
\item $\phi_j^{(k)}$ increases to $\phi_j$ a.e. wrt Lebesgue measure.
\item $\phi_j^{(k)}$ converges to $\phi_j$ uniformly on $X$.   
\end{itemize}\end{thm}
For the first operator considered, this is in fact straightforward: convergence holds locally
on the Zariski open subset $\Omega$ where weights are locally bounded
by Bedford-Taylor's results, and it extends across the boundary of $\Omega$
because the total mass is constant by Theorem~\ref{thm:mass}. The case of the second operator then follows quite easily. 

The following integration-by-parts formula is more difficult to establish. Its proof, given in~\cite{BEGZ}, is an elaboration of the Skoda-El Mir extension theorem.  

\begin{thm}
\label{thm:stokes} Let $u$ and $v$ be two bounded functions on $X$, each being a differences of quasi-psh functions that are locally bounded on a given Zariski open subset $\Omega$. Let also $\Theta$ be a closed positive
current of bidimension $(1,1)$ on $X$. Then we have 
$$\int_{\Omega}u\,dd^{c}v\wedge\Theta=\int_{\Omega}v\,dd^{c}u\wedge\Theta=-\int_\Omega dv\wedge d^cu\wedge\Theta.$$
\end{thm}

\subsection{Equilibrium weights}
Let $X$ be a compact complex manifold and $L$ be a big line bundle. Given a weighted subset $(K,\phi)$, we set 
\begin{equation}\phi_K=\sup\left\{ \psi,\,\psi\,\textrm{psh weight on$\, L$},\,\psi\leq\phi\,\,\textrm{on$\, K$}\right\},\label{eq:extem metric}\end{equation}
so that the definition (\ref{equ:equi_weight}) of the equilibrium weight $P_K\phi$ can be reformulated as
$$P_K\phi=\phi_K^*.$$
In case $K=X$ the inequality $\phi_X\le\phi$ on $X$ implies $P_X\phi\le\phi$ by continuity of $\phi$, and this means that $P_X\phi=\phi_X$ is already upper semi-continuous in that case. This property however fails for more general weighted subsets. Extending the classical terminology, a weighted subset $(K,\phi)$ will be called \emph{regular} if $\phi_K$ is usc, i.e~if $P_K\phi\le\phi$ holds on $K$. 

By Choquet's lemma (cf.~\cite{Kli} p.~38) there exists an increasing sequence of psh weights $\psi_j$ such that $\psi_j\le\phi$ on $K$ and $\lim_{j\to\infty}\psi_j=P_K\phi$ a.e. on $X$ wrt Lebesgue measure, and we can furthermore assume that the $\psi_j$ have minimal singularities by replacing them by $\max(\psi_j,\tau)$ where $\tau$ is a psh weight with minimal singularities such that $\tau\le\phi$ on $K$. 

The following `tautological maximum principle'
is a mere reformulation of the definition of $\phi_{K}$.

\begin{prop}
\label{prop:max}(Maximum principle) Let $(K,\phi)$ be weighted subset.
Then for every psh weight $\psi$ on $L$ we have 
$$\sup_{K}(\psi-\phi)=\sup_{X}(\psi-\phi_{K})$$
 In particular 
$$\Vert s\Vert_{L^{\infty}(K,k\phi)}=\Vert s\Vert_{L^{\infty}(X,k\phi_K)}$$
 for every section $s\in H^{0}(kL).$
\end{prop}
Note however that this fails with $P_K\phi=\phi_K^*$ in place of $\phi_K$ when $(K,\phi)$ is not regular. Equilibrium weights behave nicely under pull-back:

\begin{prop}\label{prop:pullback} 
Let $\pi:Y\to X$ be a surjective morphism between two compact complex manifolds of same dimension $n$, and let $L$ be a big line bundle on $X$ (so that $\pi^*L$ is also big). Let $(K,\phi)$ be a weighted subset of $(X,L)$, and consider the induced weighted subset $(\pi^{-1}K,\pi^*\phi)$ of $(Y,\pi^*L)$. Then their respective equilibrium weights are related by
$$P_{\pi^{-1}K}\pi^{*}\phi=\pi^*P_K\phi.$$
\end{prop}

We stress that $\pi$ is not assumed to have connected fibres (in which case every psh weight $\psi$ on $\pi^*L$ is of the form $\psi=\pi^*\tau$ for some psh weight on $L$).  

\begin{proof} It is clear that $P_{\pi^{-1}K}\pi^{*}\phi\ge\pi^*P_K\phi$ by definition. In order to prove the converse inequality we argue as in the proof of \cite{BEGZ} Proposition 1.12. Let $\psi$ be a psh weight on $\pi^*L$ such that $\psi\le\pi^*\phi$ on $\pi^{-1}(K)$. Let $\phi_0$ be a fixed smooth weight on $L$ and set $v:=\psi-\pi^*\phi_0$, which is a $\pi^*\theta$-psh function on $Y$ with $\theta:=dd^c\phi_0$. Define a function $u$ on $X$ by
\begin{equation}\label{equ:max}u(x):=\max_{y\in\pi^{-1}(x)}v(y).
\end{equation}
We claim that $u$ is a $\theta$-psh function. Indeed it is standard to see that $u$ is a $\theta$-psh function on the Zariski open subset $U$ of regular values of $\pi$, and one then checks that
$$u(x)=\limsup_{y\to x,\,y\in U}u(y)$$
using the fact that $v$ is quasi-psh and $\pi$ is proper, which proves the claim. Now define $\tau:=\phi_0+u$. It is a psh weight on $L$ and it easily follows from (\ref{equ:max}) that $\tau\le\phi$ on $K$, thus $\tau\le P_K\phi$. As a consequence we get $\pi^*\tau\le\pi^*P_K\phi$. On the other hand we have $\psi\le\pi^*\tau$ by (\ref{equ:max}) thus we have proved that every psh weight $\psi$ on $L$ such that $\psi\le\pi^*\phi$ on $\pi^{-1}(K)$ satisfies $\psi\le\pi^*P_K\phi$, which means that $P_{\pi^{-1}(K)}\pi^*\phi\le\pi^*P_K\phi$ as desired. 
\end{proof}

Recall from (\ref{equ:equi_meas}) that the equilibrium measure of $(K,\phi)$ is defined by
$$\eq(K,\phi):=\vol(L)^{-1}\MA(P_K\phi).$$
It is a probability measure by Theorem~\ref{thm:mass}. 
\begin{prop}
\label{prop:support} If $(K,\phi)$ is a weighted subset, then
$\eq(K,\phi)$ is concentrated on $K$
and we have $P_K\phi=\phi$ on $K$ a.e.~with respect to this measure. 
\end{prop}
The technique of proof is pretty standard (see e.g~\cite{Dem2}, p.17), but we provide details since this result plays a crucial role in the proof of Theorem B. 
\begin{proof} Let $\Omega$ be as before a Zariski open subset of $X$ such that every psh weight of $L$ with minimal singularities is locally bounded on $\Omega$. Note that $\eq(K,\phi)$ puts no mass on the Zariski closed subset $X-\Omega$ since the latter is in particular pluripolar. In order to prove (i) we thus have to show that $\eq(K,\phi)$ puts no mass on any given (small) open ball $B\subset\Omega-K$. 

By Choquet's lemma there exists a non-decreasing sequence $\psi_j$ of psh weights with minimal singularities such that $\psi_j\le\phi$ on $K$ and $\psi_j\to P_K\phi$ a.e.~wrt Lebesgue measure. Since $\psi_j$ is bounded on $B$, by Bedford-Taylor we can find a bounded psh function $\tau_j$ on $B$ such that $(dd^c\tau_j)^n=0$ and 
which coincides with $\psi_j$ on the boundary of $B$ (here we identify psh weights on $L|_B$ with psh functions, implicitly fixing a trivialization of $L|_B$). Since $\tau_j$ can be written as a Perron envelope, it follows that $\tau_j\ge\psi_j$ and $\tau_{j+1}\ge\tau_j$ on $B$. Now let $\widetilde{\psi}_j$ be the psh weight that coincides with $\psi_j$ outside $B$ and with $\tau_j$ on $B$. We then have $\widetilde{\psi}_j=\psi_j\le\phi$ on $K$ since the latter doesn't meet $B$, hence $$\psi_j\le\widetilde{\psi}_j\le\phi_K\le P_K\phi$$ 
by definition of $\phi_K$. We thus see that $P_K\phi$ is also the increasing limit a.e. of the psh weights $\widetilde{\psi}_j$. Since we have $\MA(\widetilde{\psi}_j)=0$ on $B$, it follows that $\MA(P_K\phi)=0$ on $B$ by continuity of Monge-Amp\`ere along monotonic sequences, and we have thus proved that $\eq(K,\phi)$ is concentrated on $K$. 

As a second step we prove that $\MA(P_K\phi)$ is also concentrated on the closed subset $\{P_K\phi\ge\phi\}$. The argument is essentially the same, except that we need to be slightly more careful to guarantee that $\widetilde{\psi}_j\le\phi$ on $B$. Let thus $x_0\in\Omega$ such that $P_K\phi(x_0)<\phi(x_0)-\e$ with $\e>0$. If $B$ is a small open ball centered at $x_0$, we can identify weights on $L|_B$ with functions. If $B$ is small enough we have $P_K\phi<\phi(x_0)-\e$ on $B$ by upper semi-continuity of $P_K\phi$ and $\phi\ge\phi(x_0)-\e$ by continuity of $\phi$. If $\tau_j$ denotes as above the bounded psh function on $B$ such that $(dd^c\tau_j)^n=0$ and which coincides with $\psi_j$ on the boundary of $B$, then $\psi_j\le\phi(x_0)-\e$ on $B$ implies $\tau_j\le\phi(x_0)-\e$ on the boundary of $B$, hence 
$$\tau_j\le\phi(x_0)-\e\le\phi$$ 
on $B$ by pluri-subharmonicity of $\tau_j$ (since $\phi(x_0)-\e$ is a constant). We thus see that $\widetilde{\psi}_j$ defined as above satisfies $\widetilde{\psi}_j\le\phi$ on $K$, and the same reasoning as above yields $\MA(P_K\phi)=0$ on $B$ as desired. 

Finally observe that the same sequence $\psi_j$ as above satisfies 
$$
\int_X(\psi_j-\phi)\MA(P_K\phi)\le 0
$$
since $\psi_j\le\phi$ on $K$ and $\MA(P_K\phi)$ is concentrated on $K$ by the first step of the proof. It follows that
$$\int_X(P_K\phi-\phi)\MA(P_K\phi)\le 0$$
since
$$\lim_{j\to\infty}\int_X(\psi_j-\phi)\MA(P_K\phi)=\int_X(P_K\phi-\phi)\MA(P_K\phi)$$ 
by Theorem~\ref{thm:continuous}. But we have already shown that $P_K\phi\le\phi$ a.e. wrt $\MA(P_K\phi)$, thus we get $P_K\phi=\phi$ a.e. wrt $\MA(P_K\phi)$ as desired. 
\end{proof}

We now quote from~\cite{Ber2} the following description of $\eq(X,\phi)$ for a smooth weight $\phi$ on $X$, which plays a key role in the present paper (cf.~the proof of Theorem~\ref{thm:robert} below):
\begin{thm} If $\phi$ is a smooth weight on $L$ then $dd^cP_X\phi$ has $L^\infty_{loc}$ coefficients on a Zariski open subset $\Omega$. 
\end{thm}
This result has now been extended to the case of an arbitrary big cohomology class in $H^{1,1}(X,\R)$ in~\cite{BD09}. As in~\cite{Ber2,BD09} we infer:
\begin{cor}
\label{cor:equilibrium} If $\phi$ is a smooth weight on $L$ then $\eq(X,\phi)$ is absolutely continuous with respect to Lebesgue measure. In fact we have $dd^{c}\phi\geq0$ pointwise on
the compact subset $E:=\{P_X\phi=\phi\}$, and 
$$\eq(X,\phi)=\vol(L)^{-1}{\bf 1}_E(dd^{c}\phi)^{n}.$$
\end{cor}
\begin{proof} Since $dd^cP_X\phi$ has $L^{\infty}_{\text{loc}}$ coefficients on $\Omega$ a local convolution argument shows that the Bedford-Taylor measure $(dd^cP_X\phi)^n$ has $L^{\infty}_{\text{loc}}$ density with respect to Lebesgue measure on $\Omega$ and coincides with the pointwise $n$-th exterior power of $dd^cP_X\phi$ (compare for instance~\cite{Dem2} p.16). 

If $dd^c\phi<0$ at a point $x_0\in X$ then the function $P_X\phi-\phi\le 0$ is strictly psh in a neighbourhood of $x_0$, so it cannot vanish at $x_0$ by the maximum principle. This shows that $dd^c\phi\ge 0$ pointwise on $E$. 

Since both $\eq(X,\phi)$ and $(dd^c\phi)^n$ put no mass on $X-\Omega$ there remains to show that $u:=P_X\phi-\phi$ satisfies 
$\frac{\partial u}{\partial z_i\partial\overline{z}_j}=0$ Lebesgue-a.e. on $E\cap B$ for each ball $B$ in a coordinate chart centered at a point of $E$. Since $dd^cu$ has $L^{\infty}_{\text{loc}}$-coefficients we have in particular $\Delta u\in L^1_{\text{loc}}$ hence $u\in W^{2,1}_{\text{loc}}$ by elliptic regularity. The result now follows by succesively applying Lemma~\ref{lem:deriv}Ê below to $u$ and its first partial derivatives. 
\end{proof} 

\begin{lem}\label{lem:deriv} Let $A$ be a measurable subset of $\R^m$ and let $v\in W^{1,1}_{\text{loc}}(\R^m)$ such that $v=0$ a.e. on $A$. Then $\partial v/\partial x_i=0$ a.e. on $A$ for $i=1,...,m$. 
\end{lem}
See for instance~\cite{KS} p.53 for a proof.

\subsection{Approximation by pluri-subharmonic envelopes of smooth weights}
Let $K$ be a given compact non-pluripolar subset of $X$. We first record the following straightforward properties of the projection operator $P_K$. 

\begin{lem}\label{lem:projection}The projection operator $P_K$ is non-decreasing, concave and  continuous along decreasing sequences of continuous weights on $L|_K$. It is also $1$-Lipschitz continuous: 
$$\sup_{X}|P_K\phi_1-P_K\phi_2|\leq\sup_K|\phi_1-\phi_2|$$
 for any two continuous weights $\phi_1,\phi_2$ on $L|_K$.  
\end{lem}

The following approximation result
will allow us to reduce the proof of Theorem A
to the case of smooth weights.

\begin{prop}\label{prop:approx} Let $L$ be a big line bundle.
\begin{itemize}
\item Let $\psi$ be a psh weight on $L$. Then there exists a decreasing sequence of smooth weights $\phi_j$ on $L$ such that $\lim_{j\to\infty}P_X\phi_j=\psi$
pointwise on $X$.
\item Let $(K,\phi)$ be a weighted subset. Then there exists an increasing sequence $\phi_{j}$ of smooth weights on $L$ such that $\lim_{j\to\infty} P_X\phi_j=\phi_K$
almost everywhere wrt Lebesgue measure.
\end{itemize}
\end{prop}

\begin{proof} Since $\psi$ is in particular upper semi-continuous, we can find a decreasing sequence $\phi_j$ of smooth weights such that $\phi_j\to\psi$ pointwise on $X$. Since $\psi\le\phi_j$ is psh, we infer $\psi\le P_X\phi_j\le\phi_j$, and it follows that $P_X\phi_j$ also decreases pointwise to $\psi$, which proves the first point.

Let us now prove the second point. We first claim that 
$$\phi_K=\sup\{P_X\tau,\,\,\tau\textrm{ continuous weight on }L,\,\,P_X\tau\le\phi\textrm{ on }K\}.$$
Indeed let $\psi$ be psh weights such that $\psi\le\phi$ on $K$ and let $\e>0$. By the first part of the proof, there exists a decreasing sequence $\phi_j$ of smooth weights such that $P_X\phi_j$ decreases pointwise to $\psi-\e$ as $j\to\infty$. By Dini's lemma, it follows that the usc function $P_X\phi_j-\phi$
is $\le 0$ on the compact set $K$ for $j\gg 1$
 large enough, and we thus get
$$\psi-\e\le P_X\phi_j\le\phi_K$$ 
for $j$ large enough, hence the claim. Since the family of psh weights $P_X\tau$ as above is clearly stable by max, Choquet's lemma thus shows that there exists an increasing sequence $\tau_j$ of continuous weights such that $\tau_j\le\phi$ on $K$ and $P_X\tau_j\to\phi_K$ a.e. To conclude the proof we simply take an increasing sequence of smooth weights $\phi_j$ such that
$$\tau_j-1/j\le\phi_j\le\tau_j.$$ 
\end{proof}

\begin{rem}\label{rem:ample} When $L$ is ample one can show using Demaily's regularization theorem \cite{Dem3} that the smooth weights $\phi_j$ in both parts of Proposition~\ref{prop:approx} can furthermore be taken to be \emph{strictly psh}, and in particular $P_X\phi_j=\phi_j$. This shows in particular that 
$$\phi_K=\sup\{\psi,\,\,\psi\textrm{ continuous psh weight on }L,\,\,\psi\le\phi\textrm{ on }K\},$$
which is thus always lower semi-continuous in that case. It follows that $(K,\phi)$ is regular iff $\phi_K$ is continuous when $L$ is ample, which corresponds to the classical definition (cf.~\cite{Kli}). 
\end{rem}

\section{The Bergman distortion function and the Bernstein-Markov property}\label{sec:BM}
\subsection{Bergman kernels}\label{sec:bergman}
Let $(\mu,\phi)$ be a weighted measure, and let $E$ be the support of $\mu$, which is non-pluripolar by our standing assumptions. The \emph{Bergman distortion function} $\rho(\mu,\phi)$ is defined at a point $x\in E$ as the squared operator norm of the evaluation operator
$$\mathrm{ev}_x:H^0(L)\to L_x,$$
in other words
\begin{equation}
\rho(\mu,\phi)(x)=\sup_{s\in H^{0}(L)-\{0\}}|s(x)|_\phi^2/\Vert s\Vert^2_{L^{2}(\mu,\phi)}.\label{equ:bergman}\end{equation}
Since $\mu$ is a probability measure we have
$$\Vert s\Vert_{L^{2}(\mu,\phi)}\le\Vert s\Vert_{L^\infty(E,\phi)},$$
which shows that 
$$\sup_E\rho(\mu,\phi)^{1/2}$$
is exactly the distortion between the $L^2(\mu,\phi)$ and $L^\infty(E,\phi)$-norms on $H^0(L)$. 

If $S=(s_1,...,s_N)$ denotes an $L^2(\mu,\phi)$-orthonormal basis of $H^{0}(L)$, then it is well-known that 
$$\rho(\mu,\phi)=\sum_{j=1}^N|s_j|^2_\phi.$$

The \emph{Bergman measure} associated to $(\mu,\phi)$ is now defined as 
\begin{equation}\label{equ:berg_meas}\b(\mu,\phi):=N^{-1}\rho(\mu,\phi)\mu.
\end{equation}
Note that it is a \emph{probability} measure since we have
$$\int_X\rho(\mu,\phi)\mu=\sum_j\Vert s_j\Vert^2_{L^2(\mu,\phi)}=N.$$
If we now replace $\phi$ by $k\phi$, then the relation
$$\sup_K\rho(\mu,k\phi)\ge N_k$$
shows that the distortion between the $L^2(\mu,k\phi)$ and $L^\infty(E,k\phi)$-norms on $H^{0}(kL)$ grows at least like $k^{n/2}$ as $k\rightarrow\infty$.

Assume now that $\mu$ is a smooth positive volume form on $X$ and that $\phi$ is smooth, so that $E=X$ in particular. When $\phi$ has strictly positive curvature, the celebrated Bouche-Catlin-Tian-Zelditch theorem (\cite{Bouche,Cat,Tia1,Zel}) asserts that $\beta(\mu,k\phi)$ admits a full asymptotic expansion in the space of smooth volume forms, with the probability measure $\eq(X,\phi)$ 
as the dominant term. 

As was shown by the first named author (in~\cite{Ber1} for the $\PP^n$ case and in~\cite{Ber2} for the general case), part of this result still holds when the positive curvature assumption on $\phi$ is dropped. 

\begin{thm}\label{thm:robert} Let $L$ be a big line bundle, $\mu$ be a
smooth positive volume form on $X$ and $\phi$ be a $C^2$ weight on
$L$. Then we have 
\begin{itemize}
\item  $\sup_X\rho(\mu,k\phi)=O(k^n).$
\item $\lim_{k\to\infty}\beta(\mu,k\phi)=\eq(X,\phi)$
in the weak topology of measures.  
\end{itemize}
\end{thm}
Since this result plays a crucial role in what follows, we will sketch
its proof for the convenience of the reader, and refer to~\cite{Ber2}
for the complete proof - a slightly more involved one in fact since Fujita's
theorem is not used there but rather given a direct proof by analytic
means.

\begin{proof} By an elementary argument locally replacing $\phi$ by its second order Taylor expansion at the centre of a ball and using the mean value inequality, one shows
that $\sup_X\rho(\mu,k\phi)=O(k^{n})$, i.e.~the first assertion, and 
$$\limsup_{k\to\infty}N_k^{-1}\rho(\mu,k\phi)\le\vol(L)^{-1}(dd^{c}\phi)^{n}/\mu$$
pointwise on the set where $dd^{c}\phi\ge 0$ (compare \cite{Bernsurv} Theorem 2.1). 

We now sketch the proof of the second point. Since we are dealing with probability measures, it is enough to show by weak compactness that if $\nu$ is a given accumulation point of the sequence of measures $\b(\mu,k\phi)$, then necessarily $\nu\le\eq(\mu,\phi)$.

Now set $E:=\{ P_X\phi=\phi\}$,
so that $dd^{c}\phi\geq0$ on $E$ and 
$$\eq(X,\phi)=\vol(L)^{-1}{\bf 1}_E(dd^{c}\phi)^{n}$$
by Corollary~\ref{cor:equilibrium} recalled above from~\cite{Ber2,BD09}.
Since we automatically have 
$$\rho(\mu,k\phi)\leq\exp\left(k(P_X\phi-\phi)\right)\sup_{X}\rho(\mu,k\phi)$$
by Proposition~\ref{prop:max}, the first assertion shows that $N_k^{-1}\rho(\mu,k\phi)$ tends
to $0$ (exponentially fast) pointwise on $X-E$.

Putting all this together yields 
$$\limsup_{k\to\infty}N_k^{-1}\rho(\mu,k\phi)\mu\le\eq(X,\phi)$$
a.e.~on $X$, and Lebesgue's dominated convergence finally implies
that 
$$\nu\leq\eq(X,\phi)$$ 
for any accumulation point $\nu$ as desired. 
\end{proof}

\subsection{The Bernstein-Markov property}
Let $\mu$ be a positive volume form on $X$. By the first part of Theorem~\ref{thm:robert}, if $\phi$ is a $C^2$ weight on $L$ then the distortion 
$$\sup_X\rho(\mu,k\phi)^{1/2}$$ 
between the $L^2(\mu,k\phi)$ and $L^\infty(X,k\phi)$-norms on $H^0(kL)$ grows precisely like $k^{n/2}$, which is the minimal possible growth. 

This fact is no longer true if we drop the smoothness assumption on $\phi$. Arguing as in \cite{Bernsurv} p.3 one can for instance show that the distortion is $O(k^{n/\a})$ when $\phi$ is of class $C^\a$ with $0<\a<2$, and this estimate is optimal. For general $C^0$ weights we have the following elementary fact:
\begin{lem}\label{lem:BMcont} Let $\mu$ be a smooth positive volume form. If $\phi$ is a $C^0$ weight on $L$, then the distortion has at most sub-exponential growth, i.e.~for every $\e>0$ we have 
$$\sup_X\rho(\mu,k\phi)^{1/2}=O(e^{\e k}).$$
\end{lem}
\begin{proof} Given $\e>0$ we can cover $X$ by a finite number of small balls $B_j$ on which $L$ is trivialised and $\phi$ is $\e$-close to its value $\phi_j$ at the centre of the ball. We can also assume that $X$ is still covered by smaller balls $B_j'$ relatively compact in $B_j$. A section $s\in H^0(kL)$ is identified with a holomorphic section on $B_j$, and the desired inequality 
$$|s(x)|^2e^{-2k\phi}\le Ce^{2\e k}\int_{B_j}|s|^2e^{-2k\phi}d\mu$$
for $x\in B_j'$ is thus an immediate consequence of the mean value inequality applied to the psh function $|s|^2e^{-2k\phi_j}$ on $B_j$. 
\end{proof}

We introduce the following extension of standard terminology (see~\cite{BL2}):
\begin{defi} Let $(K,\phi)$ be a weighted subset. We say that a probability measure $\mu$ on $K$ has the \emph{Bernstein-Markov property} wrt $(K,\phi)$ if the distortion between the $L^2(\mu,k\phi)$ and $L^\infty(K,k\phi)$-norms on $H^0(kL)$ has sub-exponential growth. 
\end{defi}

The following result is shown in~\cite{BBWN}, generalising results of Siciak~\cite{Sic}. 
\begin{thm}\label{thm:BM} Let $(K,\phi)$ be a weighted subset and let $\mu$ be a probability measure on $K$ putting no mass on pluripolar sets. Assume that:
\begin{itemize} 
\item $(K,\phi)$ is \emph{regular}, i.e.~$P_K\phi\le\phi$ holds on $K$.
\item $\mu$ is \emph{determining} for $(K,\phi)$, i.e. for every psh weight $\psi$ on $L$, $\psi\le\phi$ a.e.~wrt $\mu$ already implies $\psi\le\phi$ on $K$. 
\end{itemize} 
Then $\mu$ has the Bernstein-Markov property wrt $(K,\phi)$. 
\end{thm}
This somewhat technical looking criterion actually shows that a host of reasonable measures satisfy the Bernstein-Markov property. On the one hand if $K$ is for instance either (the closure of) a smoothly bounded domain or a real analytic totally real $n$-submanifold of $X$, then $(K,\phi)$ is regular for any continuous weight $\phi$. On the other hand if $(K,\phi)$ is regular then it is shown in~\cite{BBWN} that any probability measure on $K$ with support equal to $K$ is determining for $(K,\phi)$, and the domination principle (Corollary~\ref{cor:domination}) shows that the equilibrium measure of $(K,\phi)$ is determining as well (its support is equal to the \v Silov boundary of $(K,\phi)$). 

In the present article, we shall actually only use the following two special cases of Theorem~\ref{thm:BM}: either $\mu$ and $\phi$ are both smooth (in which case the Bernstein-Markov property is a trivial consequence of the mean value inequality, as already noticed), or $\mu$ is the Haar measure on the unit compact torus $\T^n\subset(\C^*)^n\subset\PP^n$ (in which case the Bernstein-Markov property was established in~\cite{NZ}). 

The next lemma will allow us to replace $L^\infty$-balls by $L^2$-balls whenever convenient. 
\begin{lem}
\label{lem:elde} Let $(K,\phi)$ be a weighted subset and let $\mu$ be a probability measure on $K$. Then we have 
$$0\le\cL_k(\mu,\phi)-\cL_k(K,\phi)\le\frac{1}{2k}\log\sup_K\rho(\mu,k\phi).$$
In particular if $\mu$ has the Bernstein-Markov property wrt $(K,\phi)$ then  
$$\lim_{k\to\infty}\cL_k(\mu,\phi)-\cL_k(K,\phi)=0.$$
\end{lem}
\begin{proof} If we set 
$$C_k:=\frac{1}{2k}\log\sup_K\rho(\mu,k\phi)$$ 
then we have 
$$\Vert s\Vert_{L^2(\mu,k\phi)}\le\Vert s\Vert_{L^\infty(K,k\phi)}\le e^{kC_k}\Vert s\Vert_{L^2(\mu,k\phi)}$$
 for all $k$ and all sections $s\in H^{0}(kL)$. Since the volume form $\vol_k$ is homogeneous of degree $2N_k=\dim_\R H^0(kL)$ on $H^0(kL)$ we get 
$$0\le\log\frac{\vol_k\cB^2(\mu,k\phi)}{\vol_k\cB^\infty(K,k\phi)}\le 2kN_k C_k$$
and the result follows by definition (\ref{equ:L-func}) and (\ref{equ:L-funcl2}) of the $\cL$-functionals. 
\end{proof}

\section{The Monge-Amp\`ere energy functional}\label{sec:energy}
In this section $L$ denotes a big line bundle on $X$. We have chosen to use the language of weights in this section since the rest of the article is naturally expressed in this language, but it is of course immediate to extend the results of this section (and Theorem B in particular) to the more general case of $\theta$-psh functions, where $\theta$ is a closed smooth $(1,1)$-form with big cohomology class.

\subsection{\label{sub:The-energy-functional}The energy functional on psh weights}
Let us first fix a psh weight $\psi_0$ with minimal singularities. As explained in the introduction, we define the Monge-Amp\`ere functional $\cE$ on psh weights with minimal singularities by the formula
$$\cE(\phi)=\frac{1}{n+1}\sum_{j=0}^n\int_X(\phi-\psi_0)\langle(dd^c\phi)^j\wedge(dd^c\psi_0^{n-j}).\rangle$$ 
This normalises $\cE$ by the condition $\cE(\psi_0)=0$. 

As in Section~\ref{sec:mixed}, the brackets denote non-pluripolar products. Concretely this means that the integrals are only extended over a Zariski open subset $\Omega$ of $X$ on which all psh weights are locally bounded, so that Bedford-Taylor wedge products are well-defined on $\Omega$. 

We now verify that $\cE$ remains a primitive of the Monge-Amp{\`e}re operator in our singular setting. 

\begin{prop}\label{prop:int} For any two psh weights $\phi_1,\phi_2$ with minimal singularities we have
$$\frac{d}{dt}_{t=0_+}\cE((1-t)\phi_1+t\phi_2)=\int_X(\phi_2-\phi_1)\MA(\phi_1).$$
\end{prop}
\begin{proof} The function $u:=\phi_2-\phi_1$ is bounded. We compute
$$(n+1)\cE((1-t)\phi_1+t\phi_2)=\int_{\Omega}(\phi_1-\psi_0+t u)\sum_{j=0}^{n}((1-t)dd^{c}\phi_1+t dd^{c}\phi_2)^j\wedge(dd^{c}\psi_0)^{n-j}$$
$$=\int_\Omega(\phi_1-\psi_0)\sum_{j=0}^n(dd^c\phi_1)^j\wedge(dd^c\psi_0)^{n-j}$$
$$+t\int_\Omega u\sum_{j=0}^n(dd^c\phi_1)^j\wedge(dd^c\psi_0)^{n-j}$$
$$+t\int_\Omega(\phi_1-\psi_0)\sum_{j=1}^nj(dd^c\phi_1)^{j-1}\wedge dd^c u\wedge(dd^c\psi_0)^{n-j}+O(t^2).$$
By integration-by-parts (Theorem~\ref{thm:stokes}) we have 
$$\int_{\Omega}(\phi_1-\psi_0)\sum_{j=1}^{n}j(dd^{c}\phi_1)^{j-1}\wedge dd^{c}u\wedge(dd^{c}\psi_0)^{n-j}=\int_{\Omega}u\,dd^{c}(\phi_1-\psi_0)\sum_{j=1}^{n}j(dd^{c}\phi_1)^{j-1}\wedge(dd^{c}\psi_0)^{n-j}$$
$$=\int_{\Omega}u\sum_{j=1}^{n}j(dd^{c}\phi_1)^{j}\wedge(dd^{c}\psi_0)^{n-j}-\int_\Omega u\sum_{j=0}^{n-1}(j+1)(dd^{c}\phi_1)^{j}\wedge(dd^{c}\psi_0)^{n-j}.$$
 Now 
$$\sum_{j=0}^{n}(dd^{c}\phi_1)^{j}\wedge(dd^{c}\psi_0)^{n-j}+\sum_{j=1}^{n}j(dd^{c}\phi_1)^{j}\wedge(dd^{c}\psi_0)^{n-j}-\sum_{j=0}^{n-1}(j+1)(dd^{c}\phi_1)^{j}\wedge(dd^{c}\psi_0)^{n-j}$$
$$=(dd^{c}\phi_1)^{n}+n(dd^{c}\phi_1)^{n}$$
 It follows that 
 $$\cE((1-t)\phi_1+t\phi_2)=\cE(\phi_1)+t\int_{\Omega}u(dd^{c}\phi_1)^{n}+O(t^2)$$
as desired.  
\end{proof}

As a consequence, we see that (\ref{equ:energy}) always holds, that is:
\begin{cor}\label{cor:cocycle} 
For any two psh weights with minimal singularities $\phi,\psi$ we have
$$\cE(\phi)-\cE(\psi)=\frac{1}{n+1}\sum_{j=0}^n\int_X(\phi-\psi)\langle(dd^c\phi)^j\wedge(dd^c\psi)^{n-j}\rangle.$$
\end{cor}
\begin{proof} We fix $\psi$ and temporarily denote by $\cF(\phi)$ the right-hand side expression. We can then apply Proposition~\ref{prop:int} with $\psi$ in place of $\phi_0$ to get 
$$\frac{d}{dt}\cF((1-t)\phi+t\psi)=\int_X(\psi-\phi)\MA((1-t)\phi+t\psi)=\frac{d}{dt}\cE((1-t)\phi+t\psi),$$
and the result follows since $\cF(\cdot)$ and $\cE(\cdot)-\cE(\psi)$ both vanish at $\psi$. 
\end{proof}

\subsection{General properties of the energy}
Theorem~\ref{thm:continuous} implies the following
continuity properties of the energy:

\begin{prop}
\label{prop:cont_en} The map $\phi\mapsto\cE(\phi)$ is continuous along converging sequences $\phi_j\to\phi$ of psh weights with minimal singularities in the following three cases.
\begin{itemize} 
\item $\phi_j$ decreases to $\phi$ pointwise.
\item $\phi_j$ increases to $\phi$ a.e. for the Lebesgue measure.
\item $\phi_j$ converges to $\phi$ uniformly on $X$. 
\end{itemize}
\end{prop}

\begin{prop}\label{prop:concave} The map $\phi\mapsto\cE(\phi)$ is non-decreasing
and concave on psh weights with minimal singularities. 
\end{prop}
\begin{proof} The first point follows from Corollary~\ref{cor:cocycle}. To prove concavity, let $\phi_1,\phi_2$ be two psh weights
with minimal singularities and set 
$$g(t):=\cE(t\phi_1+(1-t)\phi_2).$$ 
By Proposition~\ref{prop:int}, we have
$$g'(t)=\int_{X}u\,\MA((1-t)\phi_1+t\phi_2)$$
with $u:=\phi_2-\phi_1$. Computing
the second derivative yields 
$$g''(t)=n\int_{\Omega}u\,dd^{c}u\wedge((1-t)dd^{c}\phi_1+tdd^c\phi_2))^{n-1}$$
$$=-n\int_\Omega du\wedge d^cu\wedge((1-t)dd^c\phi_1+tdd^c\phi_2)^{n-1}\le 0$$
by Theorem~\ref{thm:stokes} again, and the proof is complete. 
\end{proof}

\begin{rem}
\label{rem:geodes} More generally one can consider variations along
a 1-parameter family $\phi_{t}$ (with $t$ in the unit-disc $\Delta$
in $\C$) of weights on $L$ with minimal singularities. Under suitable
regularity assumptions on $(t,x)\mapsto\phi_t(x)$ a simple modification of the
previous proof yields 
\begin{equation}dd^c_t\cE(\phi_{t})=\int_{x\in X}(dd^c_{(x,t)}\phi_t(x))^{n+1},\label{eq:prop energy is conv}\end{equation}
In the smooth case at least, this formula is well-known
in K{\"a}hler geometry. When $L$ is ample the operator that
maps a curve $\phi_{t}$ of smooth strictly psh weights to the Monge-Amp{\`e}re
measure $(dd^{c}_{(x,t)}\phi_t(x))^{n+1}$
may be identified with the geodesic curvature of the curve $dd^{c}_x\phi_{t}$
in the space of all K{\"a}hler metrics $\mathcal{K}(X,L)$ on $X$
lying in the first Chern class $c_{1}(L).$ The geodesic curvature
is defined with respect to the Riemannian metric on $\mathcal{K}(X,L)$
naturally defined at $\phi$ by taking $L^{2}$ norms of tangent vectors
with respect to the volume form $(dd^c_x\phi)^n$ \cite{Che}. Formula (\ref{eq:prop energy is conv})
thus shows that $\mathcal{E}$ is \emph{affine} along geodesic segments in $\cK(X,L)$. 

It is also interesting to note that 
$$\frac{k^{n+1}}{(n+1)!}\int_X(dd^c_{(x,t)}\phi)^{n+1}$$
is the leading term of the $(1,1)$-part of the pushed-forward form
$$\int_X\mathrm{ch}_{X\times\Delta}(kL,k\phi)\mathrm{td}(T_X,\omega),$$
which coincides with the curvature of the Quillen metric on $\det
H^\bullet(kL)$ by the main result of~\cite{BGS} (see also~\cite{Sou},
Theorem 4 p.132).  
\end{rem}

\begin{rem}\label{rem:extend} As a consequence of Proposition~\ref{prop:concave}, we may extend as in~\cite{BEGZ} the definition of $\cE(\phi)$ to an \emph{arbitrary} psh weight on $L$ as follows:
$$\cE(\phi)=\inf_{\psi\ge\phi}\cE(\psi)\in[-\infty,+\infty[$$ 
for $\psi$ ranging over all psh weights with minimal singularities such that $\psi\ge\phi$. It is straightforward to see that $\phi\mapsto\cE(\phi)$ so defined remains non-decreasing and concave on all psh weights. It is shown in~\cite{BEGZ} that it is also upper semi-continuous in the weak topology and that Corollary~\ref{cor:cocycle} remains true if both $\cE(\phi)$ and $\cE(\psi)$ are finite. 
\end{rem}

The following result relates the Monge-Amp\`ere energy $\cE_X$ on $X$ to the energy $\cE_Y$ on a hypersurface $Y$ of $X$. We assume here that $L$ is ample and $Y$ is smooth for simplicity. 
\begin{prop}\label{prop:recursion} Let $L$ be an ample line bundle, and assume that $Y$ is a smooth hypersurface of $X$ cut out by a section $s\in H^0(X,L)$. If $\phi,\psi$ are locally bounded psh weights on $L$ then we have 
$$
n\left(\cE_{Y}(\phi|_{Y})-\cE_Y(\psi|_{Y})\right)=(n+1)\left(\cE_{X}(\phi)-\cE_X(\psi)\right)+\int_{X}\log|s|_\phi\MA(\phi)-\int_X\log|s|_\psi\MA(\psi).
$$
\end{prop}
\begin{proof} Consider the following simple algebraic formula
\begin{equation}
(dd^{c}\phi)^{n}-(dd^{c}\psi)^{n}=dd^{c}((\phi-\psi)\sum_{j=0}^{n-1}(dd^{c}\phi)^{j}\wedge(dd^{c}\psi)^{n-1-j}).\label{equ:bott-chern}\end{equation}
From the point of view of Bott-Chern secondary characteristic
classes, it may be interpreted as a double transgression formula (compare
\cite{Don0,Sou}). At any rate, multiplying (\ref{equ:bott-chern}) by 
$$u_\e:=\log(|s|_\phi+\e)$$
and using integration by parts gives 
$$\int_{X}u_\e(dd^{c}\phi)^{n}-u_\e(dd^{c}\psi)^{n}+(dd^{c}\phi)\wedge(\phi-\psi)\sum_{j=0}^{n-1}(dd^{c}\phi)^{j}\wedge(dd^{c}\psi)^{n-1-j}$$
$$=\int_{X}(\phi-\psi)dd^{c}(u_\e+\phi)\wedge\sum_{j=0}^{n-1}(dd^{c}\phi)^{j}\wedge(dd^{c}\psi)^{n-1-j}.$$
Now $u_\e+\phi$ decreases to $\log|s|$ as $\e\to 0$ and $dd^c(u_\e+\phi)$ converges to the integration current $[Y]$ by the Lelong-Poincar\'e formula, and we get
$$\int_X\log|s|_\phi\MA(\phi)-\int_X\log|s|_\psi\MA(\psi)+(n+1)(\cE_X(\phi)-\cE_X(\psi))=n(\cE_Y(\phi|_Y)-\cE_Y(\psi|_Y)$$
as desired.
\end{proof}
The following pull-back formula is straightforward using Proposition~\ref{prop:pullback}. 
\begin{prop}
\label{prop:pull-back_en} Let $\pi:Y\to X$ be a surjective morphism between compact complex manifolds of same dimension $n$ and denote by $e$ its (topological) degree. Let $L$ be a big line bundle on $X$, and let $\phi,\psi$ be two psh weights with minimal singularities on $L$. Then we have
$$\mathcal{E}_{Y}(\pi^{*}\phi)-\cE_Y(\pi^{*}\psi)=e\left(\cE_{X}(\phi)-\cE_X(\psi)\right).$$
\end{prop}

\subsection{Proof of Theorem B}\label{sec:proof thmB}
In this section we prove Theorem B. Let thus $K$ be a non-pluripolar compact subset of $X$. 
We first prove that 
$$\phi\mapsto\eneq(K,\phi)=\vol(L)^{-1}\cE\circ P_K(\phi)$$ 
is concave and continuous. Concavity is an immediate consequence of Proposition~\ref{prop:concave}: for any weights $\phi_1,\phi_2$ on $L|_K$ we have
$$P_K((1-t)\phi_1+t\phi_2)\ge(1-t)P_K\phi_1+tP_K\phi_2$$
by concavity of $P_K$ (Lemma~\ref{lem:projection}) hence
$$\cE\left(P_K((1-t)\phi_1+t\phi_2)\right)\ge\cE\left((1-t)P_K\phi_1+tP_K\phi_2\right)$$
(since $\cE$ is non-decreasing)
$$\ge(1-t)\cE(P_K\phi_1)+t\cE(P_K\phi_2)$$
(since $\cE$ is concave). Continuity of $\phi\mapsto\eneq(K,\phi)$ follows from Lemma~\ref{lem:projection} and the third case of Proposition~\ref{prop:cont_en}. 

Given a continuous weight $\phi$ on $L|_K$ and $u\in C^0(K)$, the concave function $\cE\circ P_K$ 
admits a directional derivative at $\phi$ in the direction $u$, and our goal is to show that it is given by
$$\frac{d}{dt}_{t=0_+}\cE\circ P_K(\phi+tu)=\langle\lambda,u\rangle$$
where $\lambda$ is the linear form on $C^0(K)$ defined by 
$$\langle\lambda,u\rangle=\int_Ku\,\MA(P_K\phi).$$
Note that $\lambda$ computes the directional derivatives of $\cE$ at $P_K\phi$ according to Proposition~\ref{prop:int}. 

As a preliminary remark, we show:
\begin{lem}\label{lem:reg} In order to prove Theorem B one may assume that $u$ is a $C^\infty$ function on $X$.
\end{lem}
\begin{proof} Theorem B admits the following integral reformulation
$$
\cE\circ P_K(\phi+u)-\cE\circ P_K(\phi)=\int_{t=0}^1\int_K u\,\MA(P_K(\phi+tu))dt.
$$
If we let $u_j$ be a sequence of smooth functions on $X$ converging uniformly to $u$ on $K$, then $P_K(\phi+tu_j)\to P_K(\phi+tu)$ uniformly on $X$ by Lemma~\ref{lem:projection}.  
By Proposition~\ref{prop:cont_en}, we deduce that 
$$\lim_{j\to\infty}\cE\circ P_K(\phi+u_j)=\cE\circ P_K(\phi+u).$$
On the other hand for each $t$ we have 
$$\int_Ku_j\,\MA(P_K(\phi+tu_j))-\int_Ku\,\MA(P_K(\phi+tu))$$
$$=\int_K(u_j-u)\MA(P_K(\phi+tu_j))+\int_Ku\left(\MA(P_K(\phi+tu_j)-\MA(P_K(\phi+tu))\right).$$
The first term is bounded by $\vol(L)\sup_K|u_j-u|$ by Theorem~\ref{thm:mass}, whereas the second one converges to $0$ by Theorem~\ref{thm:continuous}. 
We thus see that
$$\lim_{j\to\infty}\int_Ku_j\,\MA(P_K(\phi+tu_j))=\int_Ku\,\MA(P_K(\phi+tu))$$
for all $t$, and we infer
$$\lim_{j\to\infty}\int_{t=0}^1\int_K u_j\,\MA(P_K(\phi+tu_j))dt=\int_{t=0}^1\int_K u\,\MA(P_K(\phi+tu))dt$$
by dominated convergence, which shows our claim. 
\end{proof}
From now on we will thus assume that $u$ is the restriction to $K$ of a $C^\infty$ function on $X$, that we also denote by $u$.

The problem at hand is an instance of a differentiability property for the optimal value of a concave optimisation problem with parameter. Indeed since $\cE$ is non-decreasing we have
$$\cE\circ P_K(\phi)=\sup\{\cE(\psi),\,\psi\textrm{ psh weight with minimal singularities}\,\,\psi\le\phi\textrm{ on }K\}$$
by Choquet's lemma and continuity of the energy along non-decreasing sequences. 

There has been a certain amount of work on differentiability of such optimal values in an abstract setting, but it seems that what we are trying to prove doesn't follow formally from such general results. On the other, the proof of Lemma~\ref{lem:linearize}, though pretty elementary, was inspired by more delicate considerations in~\cite{LM}.

The next lemma enables us to replace $\cE$ by its  linearisation $\lambda$ at $P_K\phi$. 
\begin{lem}\label{lem:linearize} We have
$$\frac{d}{dt}_{t=0_+}\cE\circ P_K(\phi+tu)=\frac{d}{dt}_{t=0_+}\langle\lambda, P_K(\phi+tu)-P_K\phi\rangle.$$
\end{lem}
\begin{proof} Set for simplicity
$$a:=\frac{d}{dt}_{t=0_+}\langle\lambda,P_K(\phi+tu)-P_K\phi\rangle,$$
which exists since $\lambda\circ P_K$ is concave. On the one hand, concavity of $\cE$ yields 
$$\cE\circ P_K(\phi+tu)\le\cE\circ P_K(\phi)+\langle\lambda,P_K(\phi+tu)-P_K\phi\rangle,$$
hence
$$\frac{d}{dt}_{t=0_+}\cE\circ P_K(\phi+tu)\le a.$$
On the other hand, given $\e>0$ we can fix $\de>0$ small enough such that 
\begin{equation}\label{equ:derivee}\langle\lambda, P_K(\phi+\de u)-P_K(\phi)\rangle\ge\de a-\de\e.
\end{equation}
For $t>0$ small enough we then have
$$\cE((1-t)P_K\phi+t P_K(\phi+\de u))\ge\cE\circ P_K(\phi)+t\langle\lambda,P_K(\phi+\de u)-P_K\phi\rangle-t\de\e$$
by Proposition~\ref{prop:int}
$$\ge\cE\circ P_K(\phi)+t\de a-2t\de\e$$
by (\ref{equ:derivee}). But since 
$$P_K(\phi+t\de u)\ge (1-t)P_K\phi+tP_K(\phi+\de u)$$
by concavity of $P_K$, we finally infer that
$$\cE\circ P_K(\phi+t\de u)\ge\cE\circ P_K(\phi)+t\de a-2t\de\e$$
for all $t>0$ small enough by monotonicity of $\cE$. It follows that 
$$\frac{d}{dt}_{t=0_+}\cE\circ P_K(\phi+tu)\ge a-2\e$$
for each $\e>0$, and the result follows. 
\end{proof}
We are now reduced to proving the linearised version of the problem, to wit

\begin{lem}\label{lem:linearized} The super-differential at $\phi$ of the linearised problem is reduced to $\lambda$. In other words, we have
$$\frac{d}{dt}_{t=0_+}\langle\lambda,P_K(\phi+tu)-P_K\phi\rangle=\langle\lambda,u\rangle$$
for each $u\in C^0(K)$. 
\end{lem}
Recall that the super-differential of a concave function $f$ at a point $x_0$ of an open convex subset $U$ of a locally convex topological vector space $V$ is defined as the set of all continuous linear forms $l\in V^*$ such that 
$$
f(x_0)+\langle l,x-x_0\rangle\ge f(x)
$$
for all $x\in U$, which means that $l$ defines at support hyperplane at $(x_0,f(x_0))$ to the graph of $f$ (cf.~e.g.\cite{Roc} for more details). A crucial ingredient here is the following \emph{orthogonality relation}
$$\langle\lambda, P_K\phi-\phi\rangle=0,$$
which follows from Proposition~\ref{prop:support}. Since $P_K(\phi+u)\le\phi+u$, this property implies
$$\langle\lambda,P_K(\phi+u)-P_K(\phi)\rangle\le\langle\lambda,u\rangle,$$
which means that the linear form $\lambda$ belongs to the super-differential at $0$ of the continuous concave function
$$u\mapsto\langle\lambda,P_K(\phi+u)-P_K\phi\rangle.$$

At this point, we also see that Theorem B reduces to the differentiability part of the assertion, since the differential then has to coincide with $\lambda$. 

\begin{proof} We now prove Lemma~\ref{lem:linearized}. Our goal is to show that
$$\langle\lambda,P_K(\phi+tu)-P_K\phi-tu\rangle=o(t).$$
Since on the one hand
$$P_K(\phi+tu)\le\phi+tu=P_K\phi+tu$$
$\lambda$-a.e. and on the other hand 
$$\sup_X|P_K(\phi+tu)-P_K\phi-tu|=O(t)$$
by Lemma~\ref{lem:projection}, it will be enough to show that
\begin{equation}\label{equ:comparison}\lim_{t\to 0_+}\int_{\{P_K(\phi+tu)<P_K\phi+tu\}}\lambda=0.\end{equation}
We are going to show this by applying the comparison principle (Corollary~\ref{cor:comparison}). We now fix a strictly psh weight $\phi_{+}$ with analytic singularities on $L$. Since $u$ is assumed to be smooth according to Lemma~\ref{lem:reg}, it follows that $\phi_{+}+\e u$
is psh for $\e>0$ small enough. Upon scaling $u$, we may assume that $\e=1$. 

Since $P_{K}\phi-P_{K}(\phi+tu)$ is bounded by Lemma~\ref{lem:projection}, 
it follows in particular that  
$$P_{K}(\phi)+t\phi_{+}+tu$$ and 
$$P_{K}(\phi+tu)+t\phi_{+}$$
are both psh wieights on the same $\R$-line bundle $(1+t)L$ with equivalent singularities (use the language of quasi-psh functions to make sense of the notion of psh weight on an $\R$-line bundle). The generalised comparison principle (Corollary~\ref{cor:comparison}) thus yields
$$\int_{O_t}\MA(P_K(\phi)+t(u+\phi_{+}))$$
$$\le\int_{O_t}\MA(P_{K}(\phi+tu)+t\phi_{+})$$
with 
$$O_t:=\{ P_{K}(\phi+tu)<P_{K}(\phi)+tu\}.$$
Now the binomial formula yields
\begin{equation}\label{equ:binomial}\MA(P_{K}(\phi+tu)+t\phi_{+})=\MA(P_K(\phi+tu))+\sum_{j=1}^{n}{{n \choose j}}t^{j}(dd^c P_{K}(\phi+tu))^{n-j}\wedge (dd^c\phi_{+})^{j}\end{equation}
on the Zariski open subset where all psh weights with minimal singularities are locally bounded. Since $t(u+\phi_+)$ is a psh weight on $tL$ by assumption, we have 
$$\int_{O_t}\MA(P_K\phi)\le\int_{O_t}\MA(P_K(\phi)+t(u+\phi_{+})),$$
which is in turn
$$\le\int_{O_t}\MA(P_K(\phi+tu))+O(t)$$
by (\ref{equ:binomial}) and Theorem~\ref{thm:mass}. But $P_K\phi\le\phi$ implies that
$$O_t\subset\{P_K(\phi+tu)<\phi+tu\},$$ 
and we infer that 
$$\int_{O_t}\MA(P_K(\phi+tu))=0$$
by Proposition~\ref{prop:support} again. We thus conclude that
$$\int_{O_t}\MA(P_K\phi)=O(t),$$
and the proof of Lemma~\ref{lem:linearized} is thus complete.
\end{proof} 

We now show that the energy at equilibrium is $C^{1,1}$ in the following sense: 
\begin{prop}\label{prop:c11} Let $(K,\phi)$ be a weighted subset and let $u$ be a smooth  function on $X$. Then the directional derivative of $\phi\mapsto\eneq(K,\phi)$ at $\phi$ in the direction $u$ is Lipschitz continuous with respect to the sup-norm on $K$. 
\end{prop}
\begin{proof} By Lemma~\ref{lem:projection} and Theorem B it is enough to show that 
$$\phi\mapsto\int_X u\MA(\phi)$$
is Lipschitz continuous on the space of psh weights with minimal singularities endowed with the sup-norm. By (\ref{equ:bott-chern}) above and integration by parts (Theorem~\ref{thm:stokes}) yield
$$\int_X u\MA(\phi)-\int_X u\MA(\psi)=\int_X\left(\phi-\psi\right)dd^c u\wedge\Theta$$
where the positive current
$$\Theta:=\sum_{j=1}^{n-1}\langle (dd^c\phi)^j\wedge(dd^c\psi)^{n-j}\rangle$$
has uniformly bounded mass by Theorem~\ref{thm:mass}, and the result follows.
\end{proof}

\section{Volume growth and transfinite diameter}\label{sec:thmA}

\subsection{Proof of Theorem A}
Let $(K_1,\phi_1)$ and $(K_2,\phi_2)$ be two weighted subsets. Our goal is to prove that 
\begin{equation}\label{equ:thmA}\lim_{k\to\infty}\cL_k(K_1,\phi_1)-\cL_k(K_2,\phi_2)=\eneq(K_1,\phi_1)-\eneq(K_2,\phi_2).
 \end{equation}

If this formula holds for all $(K_1,\phi_1)$ and a \emph{fixed} $(K_2,\phi_2)$, then it also holds for \emph{any} $(K_2,\phi_2)$ by taking differences. We can thus assume that $K_2=X$ and that $\phi_2$ is a fixed smooth weight on $X$.

{\bf Step 1}. As a first step, we also assume that $K_1=X$ and $\phi_1$ is smooth. Let $\mu$ be a smooth positive volume form, so that both weighted measures $(\mu,\phi_i)$, $i=1,2$ satisfy the BM property by Lemma~\ref{lem:BMcont}. By Lemma~\ref{lem:elde}, (\ref{equ:thmA}) is thus equivalent in that case to
\begin{equation}\label{equ:eldeA}\lim_{k\to\infty}\cL_k(\mu,\phi_1)-\cL_k(\mu,\phi_2)=\eneq(X,\phi_1)-\eneq(X,\phi_2).
\end{equation}
As mentioned in (\ref{equ:gram}), the volume ratio of $L^2$-balls can be expressed as a Gram determinants. As a consequence, we will prove:
\begin{lem}
\label{lem:smooth} The directional derivatives of $\cL_k(\mu,\cdot)$
at a smooth weight $\phi$ are given by integration against the Bergman measure $\b(\mu,k\phi)$. 
\end{lem}
\begin{proof} Let $v$ be a given smooth function. By (\ref{equ:gram}) we have
$$\cL_k(\mu,\phi+t v)-\cL_k(\mu,\phi)=-\frac{1}{2kN_k}\log\det H_k(t)$$
with the Gram matrix  
$$H_k(t):=\left(\int_{X}s^{(k)}_{i}\overline{s^{(k)}_{j}}e^{-2k(\phi+t v)}d\mu\right)_{1\leq i,j\leq N},$$
$S_k=(s^{(k)}_j)_j$ being a fixed orthonormal basis of $H^0(kL)$ with respect
to $L^{2}(\mu,k\phi)$. Since $H_k(0)=\id$, it follows
that 
$$\frac{d}{dt}_{t=0}\cL_k(\mu,\phi+t v)=-\frac{1}{2kN_k}\frac{d}{dt}_{t=0}\tr H_k(t)$$
$$
=-\frac{1}{2kN_k}\int_{X}\sum_{j}|s^{(k)}_{j}|^{2}(-2kv)e^{-2k\phi}\mu=\frac{1}{N_k}\int_{X}v\,\rho(\mu,k\phi)\mu$$
and the result follows by definition (\ref{equ:berg_meas}) of $\beta(\mu,k\phi)$ . 
\end{proof}
By Theorem~\ref{thm:robert} we have
\begin{equation}\label{equ:robert}\lim_{k\to\infty}\b(\mu,k\phi)=\eq(X,\phi)
\end{equation}
for any smooth weight $\phi$. Now the right-hand side is the derivative of $\eneq(X,\cdot)$ by Theorem B, so in view of Lemma~\ref{lem:smooth} we get (\ref{equ:eldeA}) by integrating (\ref{equ:robert}) along the segment between $\phi_1$ and $\phi_2$. More precisely, Lemma~\ref{lem:smooth} implies 
$$\cL_k(\mu,\phi_1)-\cL_k(\mu,\phi_2)=\int_{t=0}^{1}dt\int_{X}(\phi_1-\phi_2)\beta(\mu,k\phi_t)$$
with 
$$\phi_t:=t\phi_1+(1-t)\phi_2.$$
By (\ref{equ:robert}) we have 
$$
\int_{X}(\phi_1-\phi_2)\beta(\mu,k\phi_t)\to\int_{X}(\phi_1-\phi_2)\eq(X,\phi_t)
$$ for each $t\in[0,1]$. Since 
$$\int_X(\phi_1-\phi_2)\beta(\mu,k\phi_t)\le\sup_X|\phi_1-\phi_2|$$
for each $k$ and each $t$, it follows by dominated convergence that
$$\lim_{k\to\infty}\cL_k(\mu,\phi_1)-\cL_k(\mu,\phi_2)=\int_{t=0}^{1}dt\int_{X}(\phi_1-\phi_2)\eq(X,\phi_t)$$
$$=\int_{t=0}^{1}\frac{d}{dt}\eneq(X,\phi_t)=\eneq(X,\phi_1)-\eneq(X,\phi_2)$$
by Theorem B, as desired. 

\begin{rem} The argument just presented is similar to Donaldson's proof of Proposition 2 in~\cite{Don1}. In particular, Lemma~\ref{lem:smooth} is a variant of Lemma 2 of~\cite{Don1} (cf. also Lemma 3.1 of~\cite{Bern}). 
\end{rem}

{\bf Step 2}. We now consider the general case. We first note that 
\begin{equation}\label{equ:lk}\cL_k(K,\phi)=\cL_k(X,\phi_K)
\end{equation}
as a consequence of Proposition~\ref{prop:max}, and that $\cL_k(X,\cdot)$ is non-decreasing. By Proposition~\ref{prop:approx} we can find two sequences
$\phi_{j}^{\pm}$ of smooth weights on $L$ such that 
\begin{equation}
P_X\phi_{j}^{-}\leq\phi_K\le P_K\phi\le P_X\phi_{j}^{+}\label{equ:sequences}\end{equation}
 where $P_X\phi_{j}^{-}$ (resp.~$P_X\phi_{j}^{+}$) increases (resp.~decreases)
to $P_K\phi$ almost everywhere (resp.~everywhere) on $X$ when
$j$ tends to infinity. By Step 1, we get 
$$\cE(P_X\phi_j^-)-\eneq(X,\phi_2)=\lim_{k\to\infty}\cL_k(X,P_X\phi_j^-)-\cL_k(X,\phi_2)$$
$$\le\liminf_{k\to\infty}\cL_k(X,\phi_K)-\cL_k(X,\phi_2)\le\limsup_{k\to\infty}\cL_k(X,\phi_K)-\cL_k(X,\phi_2)$$
$$\le\lim_{k\to\infty}\cL_k(X,P_X\phi_j^+)-\cL_k(X,\phi_2)$$
by (\ref{equ:sequences}) and monotonicity of $\cL_k(X,\cdot)$ 
$$=\cE(P_X\phi_j^+)-\eneq(X,\phi_2)$$ 
by Step 1 again. Now 
$$\eneq(X,\phi_j^\pm)=\vol(L)^{-1}\cE(P_X\phi_j^\pm)$$
tends to 
$$\vol(L)^{-1}\cE(P_K\phi)=\eneq(K,\phi)$$ 
by continuity of the energy along monotonic sequences (Proposition~\ref{prop:cont_en}), and putting all this together concludes the proof of Theorem A. 

\subsection{Proof of Corollary A}
We start the proof with some algebraic preliminaries. Let $(\mu,\phi)$ be a weighted measure on $(X,L)$. For each $m\in\N$ the Hilbert space structure on $H^0(X,L)$ defined by the $L^2(\mu,\phi)$-scalar product induces a Hilbert space structure on both $H^0(X,L)^{\otimes m}$ and $H^0(X,L)^{\wedge m}$ respectively. If $(s_j)$ is an $L^2(\mu,\phi)$-orthonormal basis of $H^0(X,L)$ then $s_{i_1}\otimes...\otimes s_{i_m}$, $1\le i_1,...,i_m\le N$ and $s_{i_1}\wedge...\wedge s_{i_m}$, $1\le i_1<...<i_m\le N$, are respective orthonormal basis, which shows that the vector space embedding
$$\Psi_m:H^0(X,L)^{\wedge m}\to H^0(X,L)^{\otimes m}$$
induced by the anti-symmetrization operator
$$s_1\otimes...\otimes s_m\mapsto\sum_{\sigma\in S_m}\text{sgn}(\sigma)s_{\sigma(1)}\otimes...\otimes s_{\sigma(m)}
$$
satisfies 
\begin{equation}\label{equ:anti}\Vert\Psi_m(v)\Vert^2=m!\Vert v\Vert^2.
\end{equation}

On the other hand $H^0(X^m,L^{\boxtimes m})$ is endowed with the $L^2$-scalar product induced by the probability measure $\mu^m$ and the weight
$$(x_1,...,x_m)\mapsto\phi(x_1)+...+\phi(x_m).$$
We claim that the usual vector space isomorphism
$$H^0(X,L)^{\otimes m}\simeq H^0(X^m,L^{\boxtimes m})$$
is an isometry with respect to the Hilbert space structures. Indeed this amounts to saying that given an $L^2(\mu,\phi)$-orthonormal basis $(s_j)$ of $H^0(X,L)$ the $N^m$ sections of $H^0(X^m,L^{\boxtimes m})$ defined by
$$(x_1,...,x_m)\mapsto s_{i_1}(x_1)\otimes...\otimes s_{i_m}(x_m),\,1\le i_1,...,i_m\le N,$$
are orthonormal, which is an immediate consequence of Fubini's theorem. 

Now recall from the introduction that given a basis $S=(s_1,...,s_N)$ of $H^0(L)$ we define the determinant section $\det S\in H^0(X^N,L^{\boxtimes N})$ by
\begin{equation}\label{equ:det}(\det S)(x_1,...,x_N):=\det(s_i(x_j))_{i,j}.
\end{equation}
Given a weighted subset $(K,\phi)$ and a probability measure $\mu$ on $K$ the corresponding $L^\infty$-norm (resp. $L^2$ norm ) of $\det S$ will simply be denoted by 
\begin{equation}\label{equ:linf}
\Vert\det S\Vert_{L^\infty(K,\phi)}:=\sup_{(x_1,...,x_N)\in K^N}|\det(s_i(x_j))|e^{-\left(\phi(x_1)+...+\phi(x_N)\right)}
\end{equation}
and
\begin{equation}\label{equ:ldeux}
\Vert\det S\Vert^2_{L^2(\mu,\phi)}:=\int_{(x_1,...,x_N)\in X^N}|\det(s_i(x_j))|^2e^{-2\left(\phi(x_1)+...+\phi(x_N)\right)}\mu(dx_1)...\mu(dx_N).
\end{equation}
We will rely on the following formula for this $L^2$-norm, which is well-known in the context of determinantal point processes  (compare~\cite{Dei} p.103, \cite{Joh} Proposition 2.10) and is also familiar in quantum mechanics (Slater determinants). 

\begin{lem}\label{lem:detL2}
$$
\Vert\det S\Vert^2_{L^2(\mu,\phi)}=N!\det\left(\langle s_i,s_j\rangle_{L^2(\mu,\phi)}\right)_{i,j}.
$$
\end{lem}
\begin{proof} Let $S'$ be an $L^2(\mu,\phi)$-orthonormal basis and write
$$
s_j=\sum_{i=1}^N a_{ij}s_i'.
$$
The matrix $A=(a_{ij})$ satisfies 
$$
\det\left(\langle s_i,s_j\rangle_{L^2(\mu,\phi)}\right)_{i,j}=|\det A|^2
$$
thus (\ref{equ:ldeux}) yields
$$
\Vert\det S\Vert^2_{L^2(\mu,\phi)}=\det\left(\langle s_i,s_j\rangle_{L^2(\mu,\phi)}\right)_{i,j}\Vert\det S'\Vert^2_{L^2(\mu,\phi)}.
$$
We may thus assume that $S=S'$ is an $L^2(\mu,\phi)$-orthonormal basis and we then have to show that 
\begin{equation}\label{equ:dt}\Vert\det S\Vert^2_{L^2(\mu,\phi)}=N!
\end{equation}
But comparing definitions shows that 
$$\det S=\Psi_N(s_1\wedge...\wedge s_N)$$
where $s_1\wedge...\wedge s_N$ is a length-one generator of the determinant line
$$\det H^0(X,L):=H^0(X,L)^{\wedge N}$$
and 
$$\Psi_N:H^0(X,L)^{\wedge N}\to H^0(X,L)^{\otimes N}\simeq H^0(X^N,L^{\boxtimes N})$$
is the anti-symmetrization operator. The result now follows from (\ref{equ:anti}).
\end{proof}

Now let as in Corollary A $(E,\psi)$ be a weighted subset and $\nu$ be a probability measure with the Bernstein-Markov property for $(E,\psi)$. For each $k$ let $S_k=(s_j^{(k)})$ be an $L^2(\nu,k\psi)$-orthonormal basis of $H^0(kL)$. Given a weighted subset $(K,\phi)$ we set
$$\cD_k(K,\phi):=\frac{1}{k N_k}\log\Vert\det S_k\Vert_{L^\infty(K,k\phi)}$$
with $N_k:=h^0(kL)$ 
and our goal is thus to show that
$$\lim_{k\to\infty}\cD_k(K,\phi)=\eneq(E,\psi)-\eneq(K,\phi).$$

{\bf Step 1}. We will first show that (ii) of Corollary A is actually equivalent to (ii) of Theorem A. Let thus $\mu$ be a probability measure on $K$ with the Bernstein-Markov property for $(K,\phi)$. Since the $L^2$-norms $L^2(\mu,k\phi)$ and $L^2(\nu,k\psi)$ are induced by scalar products on $H^0(X,kL)$ the ratio of their unit-ball volumes can be expressed as a Gram determinant:
$$
\frac{\vol\cB^2(\nu,k\psi)}{\vol\cB^2(\mu,k\phi)}=\det\left(\langle s_i^{(k)},s_j^{(k)}\rangle_{L^2(\mu,k\phi)}\right)_{i,j}.
$$
By Lemma~\ref{lem:detL2} we thus get
$$
\Vert\det S_k\Vert_{L^2(\mu,k\phi)}^2=N_k!\frac{\vol\cB^2(\nu,k\psi)}{\vol\cB^2(\mu,k\phi)},
$$
or in other words
$$\frac{1}{kN_k}\log\Vert\det S_k\Vert_{L^2(\mu,k\phi)}=\cL_k(\nu,\psi)-\cL_k(\mu,\phi)+\frac{\log N_k!}{2kN_k}.$$
Now $N_k=O(k^n)$ implies 
$$\log N_k!=O(k^n\log k)=o(kN_k),$$
and we thus see that (ii) of Corollary A is equivalent to (ii) of Theorem A. 

{\bf Step 2}. We now prove (i) of Corollary A assuming that there exists a probability measure $\mu$ with the Bernstein-Markov property with respect to $(K,\phi)$ (which is not the general case). We have to show that
\begin{equation}\label{equ:bmdet}\log\Vert\det S_k\Vert_{L^\infty(K,k\phi)}=\log\Vert\det S_k\Vert_{L^2(\mu,k\phi)}+o(kN_k).
\end{equation} 
Let $\e>0$. By the Bernstein-Markov property of $\mu$ with respect to $(K,\phi)$ there exists $C>0$ such that 
\begin{equation}\label{equ:BMexpl} |s(x)|^2_{k\phi}\le Ce^{k\e}\int_X|s|^2_{k\phi}d\mu
\end{equation}
for every $k$, every section $s\in H^0(X,kL)$ and every $x\in X$. Now if  $x_1,...,x_{N_k}$ are points of $X$, then for each $j$ 
$$x\mapsto\det S_k(x_1,...,x_{j-1},x,x_{j+1},...,x_{N_k})$$
is a holomorphic section in  $H^0(X,kL)$. A successive application of (\ref{equ:BMexpl}) thus yields
$$\Vert\det S_k\Vert^2_{L^\infty(X,k\phi)}\le C^{N_k}e^{kN_k\e}\Vert\det S_k\Vert^2_{L^2(\mu,k\phi)},$$
and (\ref{equ:bmdet}) follows.

{\bf Step 3}. We finally show (i) of Theorem A for an arbitrary weighted subset $(K,\phi)$. Note that Step 2 shows in particular that (i) of Corollary A holds when $K=X$, since any smooth volume form has the Bernstein-Markov property with respect to $(X,\phi)$ by Lemma~\ref{lem:BMcont}.  We remark that $\cD_k(X,\cdot)$ is non-increasing, and a successive application of Proposition~\ref{prop:max} to each variable of the holomorphic section $\det S_k$ shows that
$$\cD_k(K,\phi)=\cD_k(X,\phi_K),$$
which is the analogue of (\ref{equ:lk}). We may then conclude by using exactly the same arguments as in Step 2 of the proof of Theorem A, simply replacing $\cL_k$ with $-\cD_k$. 

\subsection{Alternative arguments for an ample line bundle.}
The case of an \emph{ample} line bundle $L$ already covers the $\C^n$ case. For readers primarily interested in this situation we stress that all preliminary results on mixed Monge-Amp\`ere operators in Section~\ref{sec:prelim} are then standard (cf.~for instance~\cite{Dem1,Dem2}), since psh weights with minimal singularities are in fact locally bounded when $L$ is ample.

As we are going to show, somewhat simpler proofs of Theorems A and B can be provided when $L$ is ample. The main point is that Theorem A can then be obtained as direct consequence of the usual Bouche-Catlin-Tian-Zelditch theorem without relying on~\cite{Ber2,Ber3}, whereas Theorem B can be deduced by combining Theorem A with~\cite{Ber2,Ber3}. 

{\bf Proof of Theorem A}. Using the same reasoning as in Step 2 of the proof of Theorem A above, we are reduced to showing that
\begin{equation}\label{equ:amp}
\cL_k(X,\phi_1)-\cL_k(X,\phi_2)\to\cE(P_X\phi_1)-\cE(P_X\phi_2)
\end{equation}
when $\phi_1,\phi_2$ are smooth weights. By taking differences it is even enough to treat the case where $\phi_2$ is smooth and strictly psh, the existence of such a weight $\phi_2$ being guaranteed by the assumption that $L$ is ample. 

Since $P_X\phi_1$ is a continuous psh weight Richberg's regularization theorem (\cite{Ric}, see also.~\cite{Dem1} p.52) yields a sequence of smooth strictly psh weights $\psi_j$ such that
$$\e_j:=\sup_X|P_X\phi_1-\psi_j|$$
tends to $0$ as $j\to\infty$. Since $\cL_k(X,\cdot)$ is non-decreasing and satisfies the scaling property it follows that 
$$|\cL_k(X,\psi_j)-\cL_k(X,P_X\phi_1)|\le\e_j,$$
i.e. 
$$\cL_k(X,\psi_j)\to\cL_k(P_X,\phi_1)=\cL_k(X,\phi_1)$$
as $j\to\infty$ uniformly with respect to $k$. Since we also have $\cE(\psi_j)\to\cE(P_X\phi_1)$   we are thus reduced to the case where $\phi_1$ is smooth strictly psh as well. 

We now fix a smooth volume form $\mu$. Since $\mu$ has the BM property with respect to both $(X,\phi_1)$ and $(X,\phi_2)$, Lemma~\ref{lem:elde} shows that (\ref{equ:amp}) is equivalent to
$$\cL_k(\mu,\phi_1)-\cL_k(\mu,\phi_2)\to\cE(\phi_1)-\cE(\phi_2).$$
But this is just an integrated version of the Bouche-Catlin-Tian-Zelditch theorem (cf.~\cite{Bernsurv} for a particularly simple proof of a weak version suficient for our purpose). Indeed the latter says that the derivative of $\cL_k(\mu,\cdot)$ at a smooth strictly psh weight (which is equal to $\b(\mu,k\phi)$ by Lemma~\ref{lem:smooth}) converges to $\cE'(\phi)=\eq(X,\phi)$ as $k\to\infty$.

{\bf A special case of Theorem B.} Here we assume that $K=X$. If $\phi_1,\phi_2$ are smooth weights and $\mu$ is a smooth volume form Theorem A implies that 
$$\lim_{k\to\infty}\cL_k(\mu,\phi_1)-\cL_k(\mu,\phi_2)=\eneq(X,\phi_1)-\eneq(X,\phi_2).$$
On the other hand the differential of 
$\cL_k(\mu,\cdot)$ at a smooth weight $\phi$ is given by integration against $\beta(\mu,k\phi)$ by Lemma~\ref{lem:smooth} and~\cite{Ber3} (i.e. Theorem~\ref{thm:robert}) implies that 
$$\lim_{k\to\infty}\beta(\mu,k\phi)=\eq(X,\phi)$$ 
Integrating along the line segment betwen $\phi_1$ and $\phi_2$ yields
$$\eneq(X,\phi_1)-\eneq(X,\phi_2)=\int_{t=0}^1dt\int_X(\phi_1-\phi_2)\eq(X,t\phi_1+(1-t)\phi_2),$$
which is equivalent to Theorem B (for $K=X$).

\section{Applications to logarithmic pluri-potential theory}\label{sec:classical}
In this section, we will reinterpret our general results in the special case where $(X,L)=(\PP^n,\cO(1))$ and the compact subsets considered lie in the affine piece $\C^n$. As explained in the introduction, this corresponds to \emph{weighted logarithmic pluri-potential theory} in $\C^n$. 

We choose homogeneous coordinates $[Z_0:...:Z_n]$ on $\PP^n$ such that $Z_0=0$ cuts out the hyperplane at infinity, so that $z_j:=Z_j/Z_0$ define the euclidian coordinates on $\C^n$. The linear form $Z_0$ can be seen as the section in $H^0(\PP^n,\cO(1))$ inducing the constant polynomial $1$ on $\C^n$, and this section enables us to identify \emph{weights} on $\cO(1)$ over $\C^n$ to \emph{functions} on $\C^n$ by
$$\phi\mapsto v:=\phi-\log|Z_0|=-\log|Z_0|_\phi.$$
We then have $dd^c\phi=dd^c v$ on $\C^n$ by the Lelong-Poincar\'e formula, and $\phi$ extends to a psh weight (resp. with minimal singularities) on $\cO(1)$ over $\PP^n$ iff $v$ is a psh function on $\C^n$ such that $v\le\log^+|z|+O(1)$ (resp. $v=\log^+|z|+O(1)$) on $\C^n$. 

If $K$ is a compact subset of $\C^n$, $\mu$ is a probability measure on $K$ and $v\in C^0(K)$ is a continuous function on $K$, then we will talk about the weighted subset $(K,v)$ and the weighted measure $(\mu,v)$. The equilibrium weight of $(K,v)$ is then identified with the usc regularization of Siciak's extremal function attached to $(K,v)$, and will be denoted by $P_Kv$. It is thus a psh function on $\C^n$ such that $P_Kv=\log^+|z|+O(1)$. 

\subsection{Leja's transfinite diameter as an energy}
Denote by $T\subset(\C^{*})^{n}\subset\PP^{n}$ the unit compact torus induced by the toric K{\"a}hler structure of $\PP^{n}$. As is well-known, the equilibrium function of $(T,0)$ is then 
$$\max_{1\le j\le n}\log^{+}|z_j|$$
and the equilibrium measure 
$$\mu_T:=\eq(T,0)$$ 
is then the Haar probability measure on $T$. 
For each $k$, let $S_k$ denote the family of all monomials on $\C^n$ of degree at most $k$, which is an $L^2(\mu_T,0)$-orthonormal basis. Comparing definitions, Leja's transfinite diameter $d_\infty(K,v)$ (cf.~\cite{ST}) is then seen to be defined by
$$\log d_\infty(K,v)=\lim_{k\to\infty}\frac{(n+1)!}{nk^{n+1}}\log\Vert\det S_k\Vert_{L^\infty(K,kv)}$$
provided the limit exists. In the unweighted case ($v=0$), the limit has been proved to exist by Zaharjuta~\cite{Zah}. Corollary A shows that the limit also exists in the weighted case, and unravelling definitions we get

\begin{equation}\label{equ:transfinite}\log d_{\infty}(K,v)=\frac{1}{n}\sum_{j=0}^{n}\int_{\C^{n}}(\max_{i}\log^{+}|z_{i}|-P_Kv)(dd^{c}P_Kv)^{j}\wedge(dd^{c}\max_{i}\log^{+}|z_{i}|)^{n-j}.
\end{equation}
(compare~\cite{Rum,DMR} for the unweighted case).

\subsection{A weighted iterated Robin formula}

As a corollary of the recursion formula (\ref{prop:recursion}) we
get the following weighted generalisation of Rumely's Robin-type formula \cite{Rum}:

\begin{cor}
\label{cor:robin} Let $(K,v)$ be a weighted compact subset of $\C^n$. Then its transfinite
diameter satisfies 
$$\log d_{\infty}(K,v)=\frac{1}{n}\sum_{j=0}^{n}\int_{Y_{j}}(\log\left|Z_{j}\right|-P_Kv)(dd^{c}P_Kv)^{n-j}$$
 where $[Z_{0}:\dots:Z_{n}]$ denote homogeneous
coordinates in $\PP^{n}$, $Y_{0}=\PP^{n}$ and $Y_{j}=\{ Z_{0}=\dots=Z_{j-1}=0\}$
when $j\geq1$. 
\end{cor}
\begin{proof} Let $T_0:=T$, $\phi_0=\log|Z_0|$ and $\psi:=\log|Z_0|+v$. By (\ref{equ:transfinite}) we then have 
$$n\log d_{\infty}(K,v)=(n+1)\left(\cE_{Y_{0}}(P_{T_0}\phi_0)-\cE_{Y_{0}}(P_K\psi)\right)$$
with 
$$P_T\phi=\max_{0\leq j\leq n}\log|Z_j|.$$
We thus see that $P_T\phi|_{Y_{1}}$ coincides with the similarly
defined weight $P_{T_1}\phi_1$ on $Y_1$. On the other hand $|Z_{0}|_\phi\equiv 1$ on $T$
and Proposition~\ref{prop:recursion} thus implies 
$$(n+1)\left(\cE_{Y_{0}}(P_{T_0}\phi_0)-\cE_{Y_{0}}(P_K\psi)\right)=n\left(\cE_{Y_1}(P_{T_1}\phi_1)-\cE_{Y_{1}}(P_K\psi|_{Y_{1}})\right)$$
$$+\int_{Y_{0}}(\log|Z_{0}|-P_K\psi)(dd^{c}P_K\psi)^{n},$$
 and the formula follows by induction on $n$. 
\end{proof}
In case $n=1$, this formula relates the \emph{weighted Robin constant} 
$$\gamma(K,v):=\lim_{|z|\rightarrow\infty}\left(v_{K}^{*}(z)-\log|z|\right)$$
to the weighted transfinite diameter by 
$$-\log d_{\infty}(K,v)=\gamma(K,v)+\int_K(P_Kv)dd^{c}(P_Kv),$$
the weighted version of Robin's formula (cf.~\cite{ST}).

\subsection{Pull-back, the resultant and dynamics}\label{sec:resultant}
We first consider the following general dynamics situation. Let $(X,L)$ be a projective manifold endowed with an \emph{ample} line bundle and let $f:X\to X$ be an endomorphism such that 
$f^{*}L=dL$ in the Picard group of $X$ for some integer $d$, called the (first)
\emph{algebraic degree} of $f$. These assumptions imply in particular
that $f$ is a finite morphism, and its topological degree is $e=d^{n}$. We also assume that $d\geq2$,
so that $f$ is not an automorphism. 

We would like to consider the action of $d^{-1}f^*$ on the space of weights on $L$. However the equality $f^*L=dL$, which holds in $\text{Pic}(X)$, only means that $f^*L$ and $dL$ are isomorphic, and a specific choice of isomorphism is required in order to identify weights on $f^*L$ with weights on $dL$. Such a choice is equivalent to that of a lifting of $f$ to a map $F:L\to L$ that is homogeneous of degree $d$ on the fibres. 
 
The choice of a lift $F$ enables to consider the action of $d^{-1}f^*$ on weights of $L$, and the \emph{dynamical Green weight} may then be defined by
$$g_{F}:=\lim_{m\rightarrow\infty}(d^{-1}F^{*})^{m}\phi$$
 where $\phi$ is any given continuous psh weight on $L$. The Green
weight $g_{F}$ is a continuous psh weight, and is the unique fixed point
of $d^{-1}F^{*}$ in the space of continuous weights on $L$
(cf.~for instance Sibony's survey in~\cite{Sib}). The Green weight $g_F$ depends on the specific choice of a lift $F$ and not just on $f$. Indeed we have
\begin{equation}\label{equ:transfo} g_{\lambda F}=g_F+\frac{1}{d-1}\log|\lambda|
\end{equation}
for each $\lambda\in\C^*$.

Now let $(E,\phi)$ be a reference weighted subset of $X$, and define the transfinite diameter (with respect to $(E,\phi)$) of a weighted subset $(K,\psi)$ by
$$d_\infty(K,\psi):=\exp\left(\frac{n+1}{n}\left(\cE(P_E\phi)-\cE(P_K\psi)\right)\right),$$
so that this coincides with Leja's transfinite diameter for weighted compact subsets of $\C^n$ if $(E,\phi)=(T,0)$. 
We then prove the following general pull-back formula:
\begin{thm}
\label{thm:pull} There exists a constant $c>0$ such that for any weighted subset $(K,\psi)$ we have 
$$d_{\infty}(f^{-1}K,d^{-1}f^{*}\psi)=c\, d_{\infty}(K,\psi)^{1/d}$$
 and in fact 
 $$c=\exp\left(\frac{(n+1)(d-1)}{nd}\left(\cE(P_E\phi)-\cE(g_{F})\right)\right).$$ 
\end{thm}
\begin{proof}
Let $\tau$ be a psh weight with minimal singularities on $L$. We have 
$$\cE(P_E\phi)-\cE(d^{-1}f^{*}\tau)=\cE(P_E\phi)-\cE(d^{-1}f^{*}P_E\phi)+\cE(d^{-1}f^{*}P_E\phi)-\cE(d^{-1}f^{*}\tau)$$
$$=\cE(P_E\phi)-\cE(d^{-1}f^{*}P_E\phi)+d^{-1}\left(\cE(P_E\phi)-\cE(\tau)\right)$$
by Proposition~\ref{prop:pull-back_en}. On the other hand Proposition~\ref{prop:pullback} shows that
the equilibrium weight of $(f^{-1}K,d^{-1}f^{*}\psi)$ is $d^{-1}f^{*}P_K\psi$,
hence applying this to $\tau:=P_K\psi$ proves the first assertion
with 
$$c:=\exp\left(\frac{n+1}{n}\left(\cE(P_E\phi)-\cE(d^{-1}f^{*}P_E\phi)\right)\right).$$ 
On the other hand, applying the above relation to $\tau:=g_{F}$ yields
$$\cE(P_E\phi)-\cE(d^{-1}f^{*}P_E\phi)=\frac{d-1}{d}(\cE(P_E\phi)-\cE(g_{F})),$$
hence the second assertion. 
\end{proof}

We now specialise this transformation formula to $\PP^n$ and show how to recover DeMarco-Rumely's result~\cite{DMR}. 
\begin{cor}
Let $f:\PP^n\to\PP^n$ be an endomorphism of degree $d\ge 2$, and let $F:\C^{n+1}\to\C^{n+1}$ be a lifting of $f$ to a $d$-homogeneous polynomial map. Then for every weighted compact subset $(K,\psi)$ we have 
$$d_{\infty}(f^{-1}K,d^{-1}f^{*}\psi)=d_{\infty}(K,\psi)^{\frac{1}{d}}\left|\res(F)\right|^{-1/nd^{n+1}}$$
where $\res(F)$ denotes the \emph{resultant} of $F$. 
\end{cor}
\begin{proof} Our arguments mostly follow~\cite{BaBer} and~\cite{DMR} with some simplifications. The space of all $d$-homogeneous polynomial maps $F:\C^{n+1}\to\C^{n+1}$ is an affine space $\C^{N+1}$ of dimension 
$$N+1:=(n+1){{n+d \choose d}}.$$
Each such map $F$ induces a rational map $f:\PP^n\dashrightarrow\PP^n$. By~\cite{GKZ} (p.105 and p.427) there exists an \emph{irreducible} hypersurface $H$ of $\PP^N$ of degree $(n+1)d^n$ such that $f$ is an endomorphism iff $F\in\pi^{-1}(\PP^N-H)$, where $\pi:\C^{N+1}-\{0\}\to\PP^N$ denotes the quotient map.  The variety of all degree $d$ endomorphisms $f$ of $\PP^n$ is thus identified with the smooth affine variety $\PP^N-H$. The irreducible homogeneous polynomial of degree $(n+1)d^{n}$ in $N+1$ variables cutting out $H$ is called the \emph{resultant} and is denoted by $\res$. It is normalised by the condition $\res(F_0)=1$
for 
$$F_0(Z_{0},...,Z_{n})=(Z_{0}^{d},...,Z_{n}^{d}).$$

The transformation formula (\ref{equ:transfo}) above implies 
$$\cE(g_{\lambda F})=\cE(g_F)+\frac{1}{d-1}\log|\lambda|,$$
so that $F\mapsto(d-1)\cE(g_F)$ descends to a weight $\tau$ on $\cO(1)$ over $\PP^N-H$. 

The main point is now Theorem 4.5 of~\cite{BaBer}, which says that $dd^c\tau\equiv 0$ on $\PP^N-H$. On the other hand Remark 1.3 of~\cite{BaBer} implies that $\tau$ is locally bounded from above near each point of $H$, hence extends to a psh weight on $\cO(1)$ over $\PP^N$. The closed positive $(1,1)$-current $dd^c\tau$ on $\PP^N$ is supported on the irreducible hypersurface $H$, thus the Support Theorem for closed positive currents (see~\cite{Dem1} Corollary 2.14 p.165) implies 
$$dd^c\tau=c[H]$$
for some $c>0$, and in fact $c=1/(n+1)d^n$ since $H$ has degree $(n+1)d^n$. This means in turn that there exists a constant $C>0$ such that

\begin{equation}\label{equ:baber} \cE(g_F)=\frac{1}{(n+1)(d-1)d^{n}}\log|\res(F)|+C\end{equation}
for all $F$. This corresponds to Proposition~4.9 of~\cite{BaBer}, whose proof has been reformulated here. Now the Green weight of the above map $F_0$ is easily seen to be
$$g_{F_0}=\max_j\log|Z_j|,$$ 
which is also the equilibrium weight $P_{\T^n}0$ of $(\T^n,0)$. Since we have $\res(F_0)=1$, we infer that $C=\eneq(T,0)$, so that (\ref{equ:baber}) becomes
$$\exp\left(\frac{(n+1)(d-1)}{nd}\left(\cE(P_T0)-\cE(g_F)\right)\right)=|\res(F)|^{-1/nd^{n+1}},$$
and the result follows. 
\end{proof}

\section{\label{sec:equi}Analytic torsion and equidistribution of small points}

\subsection{Asymptotics of the analytic torsion}
Let $X$ be a compact K{\"a}hler manifold equiped with a fixed K{\"a}hler form
$\omega$ and induced measure $\omega^n$. If $L$ is a line bundle over $X$, recall that the complex line 
$$\det H^{\bullet}(L):=\sum_{q\ge 0}(-1)^q\det H^q(L)$$
(in our additive notation for tensor products of lines) is called the
\emph{determinant of cohomology} of $L$. If $\phi$ is a smooth weight
on $L$, then $\det H^\bullet(L)$ can be equiped with a natural $L^2$
Hermitian metric $|\cdot|_{L^2(\phi)}$, induced by the $L^2$ metric
associated with $\phi$ and the measure $\omega^n$ at the level of harmonic representatives. If $\psi$ is another smooth weight on $L$, the quotient of the corresponding $L^2$ metrics on $\det H^\bullet(L)$ yields a number  
$$\log\frac{|\cdot|^2_{L^2(\psi)}}{|\cdot|^2_{L^2(\phi)}}=\sum_{q\ge
  0}(-1)^q\log\frac{\vol\cB^2_q(\phi)}{\vol\cB^2_q(\psi)},$$
where we denote by $\cB^2_q$ the $L^2$-ball of $H^q(X,L)$ for any
$q\ge 0$. 

The \emph{Ray-Singer analytic torsion} is defined by
$$T(\phi):=\sum_{q\ge 0}(-1)^q q\log\mathrm{det}_{>0}\Delta''_q,$$
where $\Delta''_q$ denotes the anti-holomorphic Laplacian
$\overline{\partial}\overline{\partial}^*+\overline{\partial}^*\overline{\partial}$
acting on smooth $L$-valued $(0,q)$-forms on $X$, and
$\mathrm{det}_{>0}$ denotes  the zeta-regularized product of its
non-zero eigenvalues $0<\lambda_1\le\lambda_2\le...$, i.e. the
derivative at $z=0$ of the meromorphic continuation to $\C$ of the zeta-function
$\sum_j\lambda_j^{-z}$. 

The \emph{Quillen metric} on the complex line $\det H^\bullet(L)$ is then the twisted metric $$|\cdot|^2_{Q(\phi)}:=|\cdot|^2_{L^2(\phi)}e^{-T(\phi)}.$$
Theorem 1.2.3 of~\cite{BGS} (cf.~also~\cite{Sou},
Corollary 1 p.132) expresses \emph{variations} of Quillen metrics in
terms of secondary Bott-Chern forms. It implies in particular in our
case that 
\begin{equation}\label{equ:quillen}
\sum_{q\ge
  0}(-1)^q\log\frac{\vol\cB^2_q(\phi)}{\vol\cB^2_q(\psi)}+T(\phi)-T(\psi)=\int_X \widetilde{\mathrm{ch}}(\phi,\psi)\wedge\mathrm{td}(\omega)
\end{equation}
for any two smooth weights $\phi,\psi$ on $L$, where $\mathrm{td}(\omega)=1+\mathrm{Ricci}(\omega)/2+$(higher degree
terms) is the Todd form of the Hermitian bundle $(T_X,\omega)$ and
$\mathrm{\widetilde{ch}}$ denotes the secondary form of the Chern
character. Formula (\ref{equ:bott-chern}) shows that
\begin{equation}\label{equ:bott-chern_bis}\cE(\phi)-\cE(\psi)=\frac{n!}{2}\int_X\mathrm{\widetilde{ch}}(\phi,\psi).
\end{equation}

If $L$ is furthermore ample, then the higher cohomology of $kL$ vanishes for $k\gg 1$, thus (\ref{equ:quillen}) and (\ref{equ:bott-chern_bis}) imply
\begin{equation}\label{equ:quillenasym} 
\log\frac{\vol\cB^2(k\phi)}{\vol\cB^2(k\psi)}+T(k\phi)-T(k\psi)=\frac{2k^{n+1}}{n!}\left(\cE(\phi)-\cE(\psi)\right)+O(k^n).
\end{equation}
If $\phi$ is a smooth weight such that $dd^c\phi>0$ (hence $L$ is ample), the main result of~\cite{BV} is the following two-term asymptotic expansion of the analytic torsion:
$$T(k\phi)=\frac{1}{2}\int_X\log\frac{(kdd^c\phi)^n}{\omega^n}\exp(kdd^c\phi)+o(k^n)$$
$$=\frac{k^n\log k}{2(n-1)!}\vol(L)+\frac{k^n}{2n!}\int_X\log(\frac{dd^c\phi)^n}{\omega^n})(dd^c\phi)^n+o(k^n),$$
and in particular $T(k\phi)=o(k^{n+1})$. 

On the other hand, if $L$ is still ample but $\phi$ has arbitrary curvature, Theorem 10 of~\cite{BV} merely says that $T(k\phi)=O(k^{n+1})$. We will now explain how to refine this estimate using our results:

\begin{thm}\label{thm:torsion} Let $\omega$ be a K{\"a}hler metric on $X$. If $L$ is an ample line bundle and $\phi$ is a smooth weight on $L$ with arbitrary curvature, then 
$$\lim_{k\to\infty}\frac{n!}{2k^{n+1}}T(k\phi)=\cE(\phi)-\cE(P_X\phi).$$
\end{thm}

\begin{proof} Since $L$ is ample, we can choose another smooth weight $\psi$ on $L$ with
  $dd^c\psi>0$, so that $T(k\psi)=o(k^{n+1})$ by the result
  of~\cite{BV} recalled above. On the other hand Lemma~\ref{lem:elde} implies
$$\log\frac{\vol\cB^2(k\phi)}{\vol\cB^2(k\psi)}=\log\frac{\vol\cB^\infty(X,k\phi)}{\vol\cB^\infty(X,k\psi)}+o(k^{n+1}),$$ and (\ref{equ:quillenasym}) thus yields
$$\log\frac{\vol\cB^\infty(X,k\phi)}{\vol\cB^\infty(X,k\psi)}+T(k\phi)=\frac{2k^{n+1}}{n!}\left(\cE(\phi)-\cE(\psi)\right)+o(k^{n+1}).$$
Theorem A now yields the result.
\end{proof}

\begin{rem} We see that for a smooth metric on an ample line bundle Theorem A is in fact \emph{equivalent} to the above estimate for the analytic torsion. 
\end{rem}

As a consequence of their result on the asymptotics of the analytic
torsion, Bismut-Vasserot gave in Theorem 10 of~\cite{BV} an asymptotic comparison result for $L^2$ metrics induced by two different volume forms. We now give a simple proof of (a generalisation of) that result:
\begin{thm} Let $L$ be a big line bundle and $\phi$ be an arbitrary smooth weight on $L$. For any two positive measures $\mu,\nu$ on $X$, we then have 
$$\lim_{k\to\infty}\frac{1}{N_k}\log\frac{\vol\cB^2(\nu,k\phi)}{\vol\cB^2(\mu,k\phi)}=\int_X\log\left(\frac{\mu}{\nu}\right)\eq(X,\phi).$$
\end{thm}
\begin{proof} Note that if $f$ is a function on $X$ we have $\cB^2(e^{-f}\mu,\phi)=\cB^2(\mu,\phi+2f)$. Now let $f:=\log(\mu/\nu)$ and $\mu_t:=e^{-tf}\mu$ for $t\in\R$, so that $\mu_0=\mu$ and $\mu_1=\nu$. By the above remark, Lemma~\ref{lem:smooth} implies that 
$$\frac{d}{dt}\log\vol\cB^2(\mu_t,k\phi)=N_k\int_Xf\beta(\mu_t,\phi).$$
We thus get
$$\log\frac{\vol\cB^2(\nu,k\phi)}{\vol\cB^2(\mu,k\phi)}=N_k\int_{t=0}^1dt\int_Xf\beta(\mu_t,k\phi)$$
and the result follows by dominated convergence since for each $t$ we have $\beta(\mu_t,k\phi)\to\eq(X,\phi)$ by Theorem~\ref{thm:robert}.
\end{proof}

\subsection{\label{sec:heights} Adelic heights.}
Following the discussion in the introduction, let $X$ be a smooth (irreducible) projective variety over $\Q$ and $L$ be a big line bundle on $X/\Q$. Suppose given \emph{once and for all} a collection $(\phi_p)$ of continuous weights on $L_{\C_p}$ over $X(\C_p)$ for every prime $p$ such that all but finitely of them are induced by a model of $X$ over $\Z$. If $\phi$ is a continuous weight on $L_\C$, recall that 
$$\ela(\phi)=\frac{1}{kN_k}\log\vol_k\cB^\A(k\phi)$$
where $\cB^\A$ denotes the adelic unit-ball defined by~(\ref{equ:ad_ball}) 

By the adelic version of Minkowski's theorem (cf.~Appendix A of~\cite{BG}), for every $\e>0$ there exists a non-zero $s\in H^{0}(L)_\Q$ such that 
$$\log\Vert s\Vert_{L^\infty(\phi)}\le-\cL^\A_1(\phi) +\log 2+\e$$
and $\log\Vert s\Vert_{L^\infty(\phi_p)}\le 0$ for all $p$.

On the other hand, recall that the \emph{height} of a point $x\in X(\overline{\Q})$
is defined by 
\begin{equation}\label{equ:height} h^\A_{\phi}(x):=-\frac{1}{\deg(x)}\sum_{y\in Gx}\left(\log|s(y)|_{\phi}+\sum_p\log|s(y)|_{\phi_p}\right)
\end{equation}
where $G$ denotes
the absolute Galois group, $Gx$ is the (finite)
Galois orbit of $x$ and $s$ is a rational section of $L$ defined over $\Q$ such that $x$ is neither a pole nor a zero of $s$. The right-hand side of (\ref{equ:height}) is indeed independent of the choice of
$s$ by the product formula, and the sum $\sum_p$ only involves
finitely many terms. Note that $h^\A_{k\phi}(x)=k h^\A_{\phi}(x)$. 

If we use sections $s\in H^0(kL)_\Q$ provided by Minkowski's theorem to compute heights, we see by (\ref{equ:height}) that 
$$h^\A_{\phi}(x)\ge\ela(\phi)-\frac{\log 2}{k}$$
for any $x\in X(\overline{\Q})$ not in the zero divisor of $s$, where the adelic $\cL$-functionals $\ela$ are defined by~\ref{equ:ela}. As a consequence, if $x_{j}\in X(\overline{\Q})$ is a \emph{generic} sequence, i.e.~a sequence converging to the generic point of $X$ in the Zariski topology, then for each $k$ we get
$$\liminf_{j}h^\A_{\phi}(x_{j})\ge\ela(\phi)-\frac{\log 2}{k},$$
and we infer
\begin{equation}\label{equ:lowerbound}
\liminf_{j\to\infty}h^\A_{\phi}(x_{j})\ge\ena(\phi).
\end{equation}
Note that we have
$$\ena(\phi)=\frac{\widehat{\vol}(\overline{L})}{(n+1)\vol(L)}$$
in the notations of~\cite{CLT}, p.15, and (\ref{equ:lowerbound}) is thus equivalent to Lemma 5.1 of the same~\cite{CLT}. 

The main point in the proof of Theorem D is the following result.
\begin{lem}\label{lem:arith} The function $\ena(\cdot)$ 
is differentiable at any continuous weight $\phi$ such that $\ena(\phi)\in\R$. Its directional derivatives are given by integration against the equilibrium measure $\eq(X(\C),\phi)$.
\end{lem}
\begin{proof} Since the Haar measure on 
$$H^0(kL)_\A\subset H^0(kL)_\R\times\Pi_p H^0(kL)_{\Q_p}$$ 
is induced by a product measure, we see that variations of adelic $\cL$-functionals are given by
$$\ela(\psi)-\ela(\phi)=\frac{1}{kN_k}\log\frac{\vol^\R_k\cB^\infty_\R(k\psi)}{\vol^\R_k\cB^\infty_\R(k\phi)}$$
where
$$\cB^\infty_\R(\cdot):=\cB^\infty(\cdot)\cap H^0(kL)_\R$$ 
denotes the unit-ball of the sup-norm on the $\R$-vector space $H^0(kL)_\R$ of $\R$-sections, and $\vol_k^\R$ denotes Lebesgue measure on the latter space. By Lemma~\ref{lem:real} below, we get
$$\lim_{k\to\infty}\frac{1}{kN_k}\log\frac{\vol^\R_k\cB^\infty_\R(k\psi)}{\vol^\R_k\cB^\infty_\R(k\phi)}=\lim_{k\to\infty}\frac{1}{2kN_k}\log\frac{\vol_k\cB^\infty(k\psi)}{\vol_k\cB^\infty(k\phi)},
$$
where $\cB^\infty$ denotes as before the unit-ball of the sup-norm in the \emph{complex} vector space $H^0(kL)_\C$ of $\C$-sections and $\vol_k$ is Lebesgue measure on that space. We now conclude by Theorems A and B, using the trivial relation
$$\limsup_{k\to\infty}a_k-\limsup_{k\to\infty}b_k=\lim_{k\to\infty}(a_k-b_k)$$
provided the right-hand limit exists (and is finite).  
\end{proof}

\begin{lem}\label{lem:real} Let $X$ be a smooth projective variety defined over $\R$, and let $L$ be a big line bundle on $X/\R$. Let $\phi$ be a continuous weight over $X(\C)$, and denote by $\cB^\infty_\R(\phi)$ the unit-ball of the sup-norm in $H^0(L)_\R$. Then 
$$\log\frac{\vol_k^\R\left(\cB^\infty_\R(k\phi)\right)^2}{\vol\cB^\infty(k\phi)}=o(kN_k).$$
\end{lem}
\begin{proof} Let $\mu$ be a smooth positive volume form on $X(\C)$, so that $(\mu,\phi)$ has the Bernstein-Markov property. The scaling argument used in the proof of Lemma~\ref{lem:elde} immediately yields
$$\log\frac{\vol_k^\R\left(\cB^2_\R(k\phi)\right)^2}{\vol_k\cB^2(k\phi)}=\log\frac{\vol_k^\R\left(\cB^{\infty}_\R(k\phi)\right)^2}{\vol_k\cB^{\infty}(k\phi)}+o(kN_k).$$
But we can further assume that $\mu$ is invariant by complex conjugation, so that the $L^2(\mu,k\phi)$-scalar product is defined over $\R$, and it is then easy to see that the left-hand side is equal to its value in the Euclidian space situation, that is
$$\frac{\vol_k^\R\left(\cB^2_\R(k\phi)\right)^2}{\vol_k\cB^2(k\phi)}=\frac{N_k!}{\left((N_k/2)!\right)^2}$$
by expressing it in terms of Gram determinants of orthonormal basis of $H^0(L)_\R$.  
Now both $N_k$ and $N_k/2$ are $O(k^n)$, and this implies by Stirling's formula that both $\log N_k!$ and $\log(N_k/2)!$ are $O(k^n\log k)=o(k^{n+1})$. The result follows.
\end{proof}

\subsection{Proof of Theorem D}\label{sec:thmD}
If $x\in X(\overline{\Q})$ is an algebraic point, let $\mu_x$ denote the averaging measure on $X(\C)$ along the Galois orbit $Gx$. By (\ref{equ:height}) it is immediate to see that
\begin{equation}\label{equ:affine} h^\A_{\phi+v}(x)=h^\A_\phi(x)+\langle\mu_x,v\rangle
\end{equation}
for any continuous function $v$ on $X(\C)$. 

Now let $(x_j)$ be a generic sequence such that $\lim_{j\to\infty}h_\phi(x_j)=\ena(\phi)\in\R$. If $v$ is a continuous function on $X(\C)$, we are to show that 
$$\lim_{j\to\infty}\langle\mu_{x_j},v\rangle=\langle\eq(X(\C),\phi),v\rangle.$$
By Lemma~\ref{lem:arith} the right-hand side is equal to the derivative at $t=0$ of the function $g(t):=\ena(\phi+tu)$. On the other hand by (\ref{equ:affine}) the left-hand side is equal to the derivative at $t=0$ of the affine  function $f_j(t):=h_{\phi+tu}(x_j)$. The asymptotic lower bound (\ref{equ:lowerbound}) implies that
$$\liminf_{j\to\infty}f_j(t)\ge g(t)$$
for all $t$, and the following elementary lemma yields the result.

\begin{lem}
Let $f_j$ be a sequence of concave functions
on $\R$ and let $g$ be a function on $\R$ such that 
\begin{itemize}
\item  $\liminf_{j\to\infty} f_j\ge g$.
\item $\lim_{j\to\infty} f_j(0)=g(0)$.
\end{itemize}
If the $f_j$ and $g$ are differentiable at $0$, then 
$$\lim_{j\to\infty}f_j'(0)=g'(0).$$
\end{lem}
\begin{proof}
Since $f_j$ is concave, we have 
$$f_j(0)+f_j'(0)t\ge f_j(t)$$ 
and it follows that
$$\liminf_{j\to\infty} t f_j'(0)\ge g(t)-g(0).$$
The result now follows by first letting $t>0$ and then $t<0$ tend to $0$. 
\end{proof}
In other words this lemma states that if $g_j(t)=f_j(t)-g(t)$ is asymptotically minimized at $t=0$ when $j\to\infty$ in the sense that
$$g_j(t)\ge g_j(0)+o(1)$$
then the derivative at $0$ is asymptotically $0$ i.e. $g_j'(0)=o(1)$. This lemma is inspired by the variational principle in the original proof by Szpiro-Ullmo-Zhang~\cite{SUZ}. The case of concave functions $f_k$ pertains to the situation considered in \cite{BBWN}.

\end{document}